\documentclass[11pt,a4paper]{article}
\usepackage{amsmath, amscd, amssymb, amsthm, latexsym}

\theoremstyle{plain}
\newtheorem{theorem}{Theorem}[section]

\newtheorem{proposition}[theorem]{Proposition}
\newtheorem{corollary}[theorem]{Corollary}
\newtheorem{lemma}[theorem]{Lemma}
\theoremstyle{definition}
\newtheorem{definition}[theorem]{Definition}
\newtheorem{example}[theorem]{Example}

\newtheorem{note}[theorem]{Note}
\theoremstyle{remark}
\newtheorem{notation}[theorem]{\bf Notation}
\newtheorem{remark}[theorem]{\bf Remark}
\newtheorem{fact}[theorem]{\bf Fact}

\newcommand{\bbbn}{\mathbb{N}}

\newcommand{\bbbs}{\mathbb{S}}\newcommand{\bbbt}{\mathbb{T}}

\newcommand{\bbbx}{\mathbb{X}}
\newcommand{\bbby}{\mathbb{Y}}\newcommand{\bbbz}{\mathbb{Z}}
\newcommand{\tta}{{\tt a}}\newcommand{\ttb}{{\tt b}}

\newcommand{\tte}{{\tt e}}

\newcommand{\ttp}{{\tt p}}
\newcommand{\ttq}{{\tt q}}
\newcommand{\tts}{{\tt s}}\newcommand{\ttt}{{\tt t}}
\newcommand{\ttu}{{\tt u}}\newcommand{\ttv}{{\tt v}}
\newcommand{\ttx}{{\tt x}}
\newcommand{\tty}{{\tt y}}\newcommand{\ttz}{{\tt z}}

\newcommand{\couple}[2]{\mbox{$\langle #1,#2 \rangle$}}
\newcommand{\segment}{\!\upharpoonright \!}
\newcommand{\tiret}{\mbox{-}}

\newcommand{\leqct}{\leq_{\rm ct}}
\newcommand{\geqct}{\geq_{\rm ct}}

\newcommand{\eqct}{=_{\rm ct}}

\newcommand{\infct}{<_{\rm ct}}
\newcommand{\supct}{>_{\rm ct}}

\newcommand{\lless[2]}{\ \rm \ll^{#1}_{#2}\ }
\newcommand{\ggreater[2]}{\ \rm \gg^{#1}_{#2}\ }

\newcommand{\alphabet}{{\bf 2}} 
\newcommand{\words}{\alphabet^*}

\newcommand{\diff}{\mbox{\em diff}}

\newcommand{\sizesub}{\footnotesize}

\newcommand{\submax}{{\mbox{{\rm\sizesub max}}}}

\newcommand{\submin}{{\mbox{{\rm\sizesub min}}}}

\newcommand{\PR[1]}{{PR^{#1}}}
\newcommand{\maxr[1]}{{Max^{#1}_{Rec}}}
\newcommand{\maxrA[1]}{{Max^{#1}_{Rec^A}}}
\newcommand{\maxrAj[1]}{{Max^{#1}_{Rec^{A'}}}}
\newcommand{\maxpr[1]}{{Max^{#1}_{PR}}}
\newcommand{\maxprA[1]}{{Max^{#1}_{PR^A}}}
\newcommand{\maxprAj[1]}{{Max^{#1}_{PR^{A'}}}}

\newcommand{\minr[1]}{{Min^{#1}_{Rec}}}
\newcommand{\minrA[1]}{{Min^{#1}_{Rec^A}}}
\newcommand{\minpr[1]}{{Min^{#1}_{PR}}}
\newcommand{\minprA[1]}{{Min^{#1}_{PR^A}}}
\newcommand{\minprjump[1]}{{Min_{PR^{#1}}}}
\newcommand{\kmax[1]}{{K^{#1}_\submax}}
\newcommand{\kmin[1]}{{K^{#1}_\submin}}
\newcommand{\kmaxbiblio}{{K_\submax}}
\newcommand{\kminbiblio}{{K_\submin}}

\newcommand{\card}{\mathit{card}}
\newcommand{\indexx}{\mathit{index}}
\newcommand{\KleenePC}{\mathit{PC}}
\newcommand{\EffCont}{\mathit{EffCont}}
\newcommand{\Eff}{\mathit{Eff}}
\newcommand{\Church}{\mathit{Church}}
\newcommand{\church}{\mathit{church}}
\begin{document}
\title{Kolmogorov Complexity\\ and\\
Set theoretical Representations of Integers, I}
%
\author{\small\sc Marie Ferbus-Zanda\\
{\footnotesize LIAFA, Universit\'e Paris 7}\\
{\footnotesize 2, pl. Jussieu 75251 Paris Cedex 05}\\
{\footnotesize France}\\
{\footnotesize\tt ferbus@logique.jussieu.fr}
\and {\small\sc Serge Grigorieff}\\
{\footnotesize LIAFA, Universit\'e Paris 7}\\
{\footnotesize 2, pl. Jussieu 75251 Paris Cedex 05}\\
{\footnotesize France}\\
{\footnotesize\tt seg@liafa.jussieu.fr}}
\date{\today}
\maketitle
{\footnotesize \textnormal \tableofcontents}
%
%
\begin{abstract}
We reconsider some classical natural semantics of integers
(namely iterators of functions,
cardinals of sets,
index of equivalence relations)
in the perspective of Kolmogorov complexity.
To each such semantics one can attach a simple representation of
integers that we suitably effectivize in order to develop an
associated Kolmogorov theory.
Such effectivizations are particular instances of a general notion
of ``self-enumerated system" that we introduce in this paper.
Our main result asserts that, with such effectivizations,
Kolmogorov theory allows to quantitatively distinguish the
underlying semantics.
We characterize the families obtained by such effectivizations
and prove that the associated Kolmogorov complexities constitute
a hierarchy which coincides with that of Kolmogorov complexities
defined via jump oracles and/or infinite computations
(cf. \cite{ferbusgrigoKmaxKmin}).
This contrasts with the well-known fact that usual Kolmogorov
complexity does not depend (up to a constant) on the chosen
arithmetic representation of integers, let it be in any base
$n\geq 2$ or in unary.
Also, in a conceptual point of view, our result can be seen as
a mean to measure the degree of abstraction of these diverse
semantics.
\end{abstract}
%
\section{Introduction}        \label{s:intro}
%
\begin{notation}\label{not:leqct}
Equality, inequality and strict inequality up to a constant
between total functions $D\to\bbbn$, where $D$ is any set,
are denoted as follows:
\begin{eqnarray*}
f\ \leqct\ g &\Leftrightarrow&
\exists c\in\bbbn\ \forall x\in D\ f(x)\leq g(x)+c
\\
f\ \eqct\ g &\Leftrightarrow &
f\leqct g\ \wedge\ g\leqct f\\
&\Leftrightarrow &
\exists c\in\bbbn\ \forall x\in D\ |f(x)-g(x)|\leq c
\\
f\ \infct\ g &\Leftrightarrow &
f \leqct g\ \wedge\ \neg(g \leqct f)\\
&\Leftrightarrow &
f \leqct g\ \wedge\ \forall c\in\bbbn\ \exists x\in D\ g(x)>f(x)+c
\end{eqnarray*}
\end{notation}
As we shall consider $\bbbn$-valued partial functions
with domain $\bbbn$, $\bbbz$, $\words$, $\bbbn^2$,..., the following
definition is convenient.
\begin{definition}\label{def:basic}
A basic set $\bbbx$ is any non empty finite product of sets
among $\bbbn,\bbbz$ or the set $\words$ of finite binary words
or the set $\Sigma^*$ of finite words in some finite or countable
alphabet $\Sigma$.
\end{definition}
Let's also introduce some notations for partial recursive
functions.
\begin{notation}\label{not:PR}
Let $\bbbx,\bbby$ be basic sets.
We denote $\PR[\bbbx\to\bbby]$ (resp. $\PR[A,\bbbx\to\bbby]$)
the family of partial recursive (resp.partial $A$-recursive)
functions $\bbbx\to\bbby$.
In case $\bbbx=\bbby=\bbbn$,we simply write $\PR[]$ and $\PR[A]$.
\end{notation}
%
%
\subsection{Kolmogorov complexity and representations of
$\bbbn$, $\bbbz$}
\label{ss:mainResults}
%
Kolmogorov complexity $K:\bbbn\to\bbbn$ maps an integer $n$ onto
the length of any shortest binary program $\ttp\in\words$ which
outputs $n$.
The invariance theorem asserts that, up to an additive constant,
$K$ does not depend on the program semantics $\ \ttp\mapsto n\ $,
provided it is a universal partial recursive function.
\\
As a straightforward corollary of the invariance theorem, $K$ does
not depend (again up to a constant) on the representation of
integers, i.e. whether the program output $n$ is really in $\bbbn$
or is a word in some alphabet $\{1\}$ or $\{0,...,k-1\}$, for some
$k\geq2$, which gives the unary or base $k$ representation of $n$.
A result which is easily extended to all partial recursive
representations of integers, cf. Thm.\ref{thm:recrep}.
\medskip\\
{\em In this paper, we show that this is no more the case when (suitably
effectivized) classical set theoretical representations are
considered.}
We particularly consider representations of integers via
\begin{itemize}
\item
Church iterators (Church \cite{church33}, 1933),
\item
cardinal equivalence classes
(Russell \cite{russell08} \S IX, 1908, cf. \cite{heijenoort} p.178),
\item
index equivalence classes.
\end{itemize}
Following the usual way to define $\bbbz$ from $\bbbn$, we also
consider representations of a relative integer $z\in\bbbz$
as pairs of representations of non negative integers $x,y$
satisfying $z=x-y$.
In the particular case of Church iterators, restricting to injective
functions and considering negative iterations, leads to another
direct way of representing relative integers.
\medskip\\
Programs are at the core of Kolmogorov theory. They do not work
on abstract entities but require formal representations of objects.
Thus, we have to define effectivizations of the above abstract set
theoretical notions in order to allow their elements to be computed
by programs.
To do so, we use computable functions and functionals
and recursively enumerable sets.
\medskip\\
Effectivized representations of integers constitute particular
instances of {\em self-enumerated representation systems}
(cf. Def.\ref{def:self}).
This is a notion of family ${\cal F}$ of partial functions
from $\words$ to some fixed set $D$ for which an invariance theorem
can be proved using straightforward adaptation of original
Kolmogorov's proof.
Which leads to a notion of Kolmogorov complexity
$K_{\cal F}^D:D\to\bbbn$, cf. Def.\ref{def:Kself}.
The ones considered in this paper are
$$K_{\Church}^\bbbn\ ,\ K_{\Church}^\bbbz\ ,
\ K_{\Delta \Church}^\bbbz\ ,\
  K_{\card}^\bbbn\ ,\   K_{\Delta \card}^\bbbz\ ,\
  K_{\indexx}^\bbbn\ ,\  K_{\Delta \indexx}^\bbbz$$
associated to the systems obtained by effectivization of the
Church, cardinal and index representations of $\bbbn$ and
the passage to $\bbbz$ representations as outlined above.
\medskip\\
The main result of this paper states that the above Kolmogorov
complexities coincide (up to an additive constant) with those
obtained via oracles and infinite computations as introduced in
\cite{becherchaitindaicz}, 2001, and
our paper \cite{ferbusgrigoKmaxKmin}, 2004.
\begin{theorem}[Main result]\label{thm:A}$\medskip\\ $
\centerline{$\begin{array}{rclcrcl}
K_{\Church}^\bbbn
&\eqct& K_{\Church}^\bbbz\segment\bbbn
&\eqct& K_{\Delta \Church}^\bbbz\segment\bbbn
&\eqct& K
\medskip\\
K_{\card}^\bbbn &\eqct& \kmax[]
&& K_{\Delta \card}^\bbbz\segment\bbbn &\eqct& K^{\emptyset'}
\medskip\\
K_{\indexx}^\bbbn &\eqct&  \kmax[\emptyset']
&& K_{\Delta \indexx}^\bbbz\segment\bbbn
&\eqct& K^{\emptyset''}
\end{array}$}
\end{theorem}
Thm.\ref{thm:A} gathers the contents of
Thms.\ref{thm:card}, \ref{thm:DeltaCard},
\ref{thm:index}, \ref{thm:DeltaIndex},
\ref{thm:Church}, \ref{thm:DeltaChurch} and \S\ref{ss:churchZ}.
\\
A preliminary ``light" version of this result was presented in
\cite{ferbusgrigoClermont}, 2002.
\medskip\\
The strict ordering result $K\supct\kmax[]\supct K^{\emptyset'}$
(cf. Notations \ref{not:leqct})
proved in \cite{becherchaitindaicz,ferbusgrigoKmaxKmin}
and its obvious relativization (cf. Prop.\ref{p:degrees})
yield the following hierarchy theorem.
\begin{theorem}\label{thm:B}
$$\log \supct
\begin{array}{c}
    K_{\Church}^\bbbn\\
    \eqct\\
    K_{\Church}^\bbbz\segment\bbbn\\
    \eqct\\
    K_{\Delta \Church}^\bbbz\segment\bbbn
\end{array}
\supct K_{\card}^\bbbn
\supct K_{\Delta \card}^\bbbz\segment\bbbn
\supct K_{\indexx}^\bbbn
\supct K_{\Delta \indexx}^\bbbz\segment\bbbn$$
\end{theorem}
This hierarchy result for set theoretical representations somewhat
reflects their {\em degrees of abstraction}.
\medskip\\
Though Church representation via iteration functionals can be
considered as somewhat complex, we see that, surprisingly, the
associated Kolmogorov complexities collapse to the simplest
possible one.
\medskip\\
Also, it turns out that, for cardinal and index representations,
the passage from $\bbbn$ to $\bbbz$, i.e. from $K_\card^\bbbn$ to
$K_{\Delta \card}^\bbbz$ and from $K_\indexx^\bbbn$ to
$K_{\Delta \indexx}^\bbbz$  does add complexity.
However, for Church iterators, the passage to $\bbbz$ does not modify
Kolmogorov complexity, let it be via the $\Delta$ operation
(for $K_{\Delta \Church}^\bbbz$) or restricting iterators to
injective functions (for $K_{\Church}^\bbbz$).
\medskip\\
The results about the $\Delta card$ and $\Delta index$ classes
are corollaries of those about the $card$ and $index$ classes and of
the following result (Thm.\ref{thm:Deltamax}) which gives a simple
normal form to functions computable relative to a jump oracle,
and is interesting on its own.
\begin{theorem}\label{thm:ADelta}
Let $A\subseteq\bbbn$.
A function $G:\words\to\bbbz$ is partial $A'$-recursive if and only
if there exist total $A$-recursive functions
$f,g:\words\times\bbbn\to\bbbn$ such that, for all $\ttp$,
$$G(\ttp)=\max\{f(\ttp,t):t\in\bbbn\}-\max\{g(\ttp,t):t\in\bbbn\}$$
(in particular, $G(\ttp)$ is defined if and only if both $\max$'s
are finite).
\end{theorem}
%
%
\subsection{Kolmogorov complexities and families of functions}
\label{ss:thmC}
%
The equalities in Thm.\ref{thm:A} are, in fact, corollaries of
equalities between families of functions $\words\to\bbbn$
(namely, the associated self-enumerated representation systems,
cf. \S\ref{ss:self}) which are interesting on their own.
For instance (cf. Thms.\ref{thm:card}, \ref{thm:DeltaCard},
\ref{thm:index}, \ref{thm:DeltaIndex},
\ref{thm:Church}, \ref{thm:DeltaChurch} and \S\ref{ss:churchZ}),
\begin{theorem}\label{thm:C}
Denote $X\to Y$ the class of {\em partial} functions from $X$ to $Y$.
\\
{\bf 1.}
A function $f:\words\to\bbbn$ is the restriction to a $\Pi^0_2$ set
of a partial recursive function if and only if
it is of the form $f=\Church\circ \Phi$ where
\\ - $\Phi:\words\to (\bbbn\to\bbbn)^{(\bbbn\to\bbbn)}$
is a computable functional,
\\ - $\Church:(\bbbn\to\bbbn)^{(\bbbn\to\bbbn)}\to\bbbn$ is the
functional such that
\begin{eqnarray*}
\Church(\Psi)&=&\left\{
\begin{array}{ll}
n&\mbox{if $\Psi$ is the iterator }f\mapsto f^{(n)}\\
\mbox{undefined}&\mbox{otherwise}
\end{array}\right.
\end{eqnarray*}
{\bf 2.}
A function $f:\words\to\bbbn$ is the $\max$ of a total recursive
(resp. total $\emptyset'$-recursive) sequence of functions
(cf. Def.\ref{not:maxpr})
if and only if it is of the form
$$\ttp\mapsto \card(W_{\varphi(\ttp)}^\bbbn)\ \ \
\mbox{ (resp. }
\ttp\mapsto
\indexx(W_{\varphi(\ttp)}^{\bbbn^2})\mbox{, up to $1$)}$$
for some total recursive $\varphi:\words\to\words$, where
\\ - $W_\ttq^\bbbn$ (resp. $W_\ttq^{\bbbn^2}$) is the r.e.
subset of $\bbbn$ (resp. $\bbbn^2$) with code $\ttq$,
\\ - $\card:P(\bbbn)\to\bbbn$ is the cardinal function
(defined on the sole finite sets),
\\ - $\indexx:P(\bbbn^2)\to\bbbn$ is defined on equivalence relations
with finitely many classes and gives the index (i.e. the number
of equivalence classes).
\medskip\\
{\bf 3.}
A function $f:\words\to\bbbn$ is partial $\emptyset'$-recursive
(resp. $\emptyset''$-recursive) if and only if it is of the form
$$\ttp\mapsto \card(W_{\varphi_1(\ttp)}^\bbbn)
             -\card(W_{\varphi_2(\ttp)}^\bbbn)
\
\mbox{ (resp. } \ttp\mapsto
\indexx(W_{\varphi_1(\ttp)}^{\bbbn^2})
-\indexx(W_{\varphi_2(\ttp)}^{\bbbn^2})\mbox{)}$$
for some total recursive $\varphi_1,\varphi_2:\words\to\words$.
\end{theorem}
%
%
%
\subsection{Road map of the paper}
\label{ss:road}
%
\S\ref{s:Kabstract} introduces the notion of self-enumerated
representation system with its associated Kolmogorov complexity.
\\
\S\ref{s:operations} introduce simple operations
on self-enumerated systems.
\\
\S\ref{s:Delta} sets up some connections between self-enumerated
representation systems for $\bbbn$ and $\bbbz$.
\\
\S\ref{s:RE} considers a self-enumerated representation system for
the set of recursively enumerable subsets of $\bbbn$.
\\
\S\ref{s:InfiniteComp} recalls material from
Becher \& Chaitin \& Daicz, 2001 \cite{becherchaitindaicz}
and our paper \cite{ferbusgrigoKmaxKmin}, 2004,
about some extensions of Kolmogorov complexity involving infinite
computations.
This is to make the paper self-contained.
\\
\S\ref{s:abstract} introduces abstract representations and their
effectivizations.
\\
\S\ref{s:card}, \ref{s:index}, \ref{s:church} develop the
set-theoretical representations mentioned in \S\ref{ss:mainResults}
and prove all the mentioned theorems and some more results
related to the associated self-enumerated systems, in particular
the syntactical complexity of universal functions for such systems.
%
%
\section{An abstract setting for Kolmogorov complexity:
self-enumerated representation systems}
\label{s:Kabstract}
%
%
\subsection{Classical Kolmogorov complexity}
\label{ss:K}
%
Classical Kolmogorov complexity of elements of a basic set $\bbbx$
is defined as follows (cf. Kolmogorov, 1965 \cite{kolmo65}):
\begin{enumerate}
\item
To every $\varphi:\words\to\bbbx$ is associated
$K^\bbbx_\varphi:\bbbx\to\bbbn$
such that
\\\centerline{$K^\bbbx_\varphi(\ttx)
=\min\{|\ttp|:\varphi(\ttp)=\ttx\}$}
i.e. $K^\bbbx_\varphi(\ttx)$ is the shortest length of a ``program"
$\ttp\in\words$ which is mapped onto $\ttx$ by $\varphi$.
\item
Kolmogorov Invariance Theorem asserts that, letting $\varphi$
vary in $\PR[\words\to\bbbx]$ (cf. Notation \ref{not:PR}),
there is a least $K^\bbbx_\varphi$, up to an additive constant:
\\\centerline{$\exists\varphi\in \PR[\words\to\bbbx]\ \
               \forall\psi\in \PR[\words\to\bbbx]\ \
             \ K^\bbbx_\varphi\leqct K^\bbbx_\psi$}
Kolmogorov complexity $\ K_\bbbx:\bbbn\to\bbbn\ $ is such
a least $K^\bbbx_\varphi$, so that it is defined up to an additive
constant.
\end{enumerate}
Let $A\subseteq\bbbn$. The above construction relativizes to oracle
$A$ : replace $\PR[\words\to\bbbx]$ by $\PR[A,\words\to\bbbx]$
to get the oracular Kolmogorov complexity $K_\bbbx^A$.
%
\subsection{Self-enumerated representation systems}
\label{ss:self}
%
We introduce an abstract setting for the definition of Kolmogorov
complexity: {\em self-enumerated representation systems}.
As a variety of Kolmogorov complexities is considered, this
allows to unify the multiple variations of the invariance
theorem, the proofs of which repeat, mutatis mutandis, the same
classical proof due to Kolmogorov
(cf. Li \& Vitanyi's textbook \cite{livitanyi} p.97).
\\
This abstract setting also leads to a study of operations on
self-enumerated systems, some of which are presented in
\S\ref{s:Delta},\ref{s:RE} and some more are developed in the
continuation of this paper.
\\
Some intuition for the next definition is given in Note
\ref{note:intuition} and Rk.\ref{rk:self}.
\begin{definition}[Self-enumerated representation systems]
\label{def:self}$\\ $
{\bf 1.} A self-enumerated representation system
(in short ``self-enumerated system") is a pair $(D,{\cal F})$
where $D$ is a set --- the domain of the system ---
and ${\cal F}$ is a family of partial functions $\words\to D$
satisfying the following conditions:
\begin{enumerate}
\item[i.]
${\displaystyle D=\bigcup_{F\in{\cal F}} Range(F)}$,
i.e. every element of $D$ appears in the range of
some function $F\in{\cal F}$.
\item[ii.]
If $\varphi:\words\to\words$ is a recursive total function
and $F\in{\cal F}$ then $F\circ \varphi\in{\cal F}$.
\item[iii.]
There exists $U\in{\cal F}$
(called a universal function for ${\cal F}$) and a total
recursive function $comp_U:\words\times\words\to\words$
such that
$$\forall F\in{\cal F}\ \ \exists \tte\in\words\ \
\forall \ttp\in\words\ \ F(\ttp)=U(comp_U(\tte,\ttp))$$
In other words, letting
$U_\tte(\ttp)=U(comp_U(\tte,\ttp))$,
the sequence of functions $(U_\tte)_{\tte\in\bbbn}$
is an enumeration of ${\cal F}$.
\end{enumerate}
{\bf 2.} {\bf (Full systems)}
In case condition ii holds for all {\em partial
recursive functions} $\varphi$, the system $(D,{\cal F})$
is called a self-enumerated representation {\em full system}.
\medskip\\
{\bf 3.}
{\bf (Good universal functions)} A universal function
$U$ for ${\cal F}$ is good if
its associated $comp$ function satisfies the condition
$$\forall\tte\ \exists c_\tte\ \forall\ttp\
 |comp_U(\tte,\ttp)|) \leq |\ttp| +c_\tte$$
i.e. for all $\tte$, we have
$(\ttp\mapsto |comp_U(\tte,\ttp)|) \leqct |\ttp|$
(cf. Notation \ref{not:leqct}).
\end{definition}
\begin{note}[Intuition]\label{note:intuition}$\\ $
{\em 1.}
The set $\words$ is seen as a family of programs
to get elements of $D$.
The choice of binary programs is a fairness condition in view
of the definition of Kolmogorov complexity (cf. Def.\ref{def:Kself})
based on the length of programs:
larger the alphabet, shorter the programs.
\medskip\\
{\em 2.}
Each $F\in{\cal F}$ is seen as a programming language with
programs in $\words$.
Special restrictions: no input, outputs are elements of $D$.
\medskip\\
{\bf 3.}
Denomination $comp$ stands for ``compiler" since it maps
a program $\ttp$ from ``language" $F$ (with code $\ttp$) to its
$U$-compiled form $comp_U(\tte,\ttp)$ in the ``language" $U$.
\medskip\\
{\bf 4.}
``Compilation" with a good universal function does not
increase the length of programs but for some additive constant
which depends only on the language, namely on the sole code $e$.
\end{note}
\begin{example}\label{ex:self}
If $\bbbx$ is a basic set then $(\bbbx,\PR[\words\to\bbbx])$ is
obviously a self-enumerated representation system.
\end{example}
\begin{remark}\label{rk:self}
In view of the enumerability condition {\em iii} and
since there is no recursive enumeration of total recursive
functions, one would a priori rather require condition {\em ii}
to be true for all partial recursive functions
$\varphi:\words\to\words$, i.e. consider the sole full systems.
\\
However, there are interesting self-enumerated representation
systems which are not full systems.
The simplest one is $\maxr[]$, cf. Prop.\ref{p:maxprmaxr}.
Other examples we shall deal with involve higher order domains
consisting of infinite objects, for instance the domain
$RE(\bbbn)$ of all recursively enumerable subsets of $\bbbn$,
cf. \S\ref{ss:RE}.
{\em The partial character of computability is already inherent to
the objects in the domain or to the particular notion of
computability and an enumeration theorem does hold for a family
${\cal F}$ of total functions.}
\end{remark}
From conditions i and iii of Def.\ref{def:self},
we immediately see that
\begin{proposition} \label{p:onto}
Let $(D,{\cal F})$ be a self-enumerated system.
Then $D$ and ${\cal F}$ are countable and any universal function
for ${\cal F}$ is surjective.
\end{proposition}
Another consequence of condition iii of Def.\ref{def:self}
is as follows.
\begin{proposition} \label{p:univ}
Let $(\bbbn,{\cal F})$ be a self-enumerated system.
Then all universal functions for ${\cal F}$ are many-one equivalent.
\end{proposition}
%
%
\subsection{Good universal functions always exist} \label{ss:good}
%
Let's recall a classical way to code pairs of words.
\begin{definition}[Coding pairs of words]\label{def:pair}$\\ $
Let $\mu:\words\to\words$ be the morphism
(relative to the monoid structure of concatenation product on words)
such that $\mu(0)=00$ and $\mu(1)=01$.
\\
The function $c:\words\times\words\to\words$ such that
$c(\tte,\ttp)=\mu(\tte)1\ttp$ is a recursive injection
which satisfies equation
\begin{equation}\label{eq:c}
              |c(\tte,\ttp)|=|\ttp|+2|\tte|+1
\end{equation}
Denoting $\lambda$ the empty word, we define
 $\pi_1, \pi_2:\words\to\words$ as follows:
\medskip\\\medskip\centerline{$\pi_1(c(\tte,\ttp))=\tte\ ,\
                               \pi_2(c(\tte,\ttp))=\ttp\ ,\
   \pi_1(w)=\pi_2(w)=\lambda \mbox{ if } w\notin Range(c)$}
\end{definition}
\begin{remark}
If we redefine $c$ as $c(\tte,\ttp)=\mu(Bin(|\tte|))1\tte\ttp$
where $Bin(k)$ is the binary representation of the integer
$k\in\bbbn$ then equation (\ref{eq:c}) can be sharpened to
$$|c(\tte,\ttp)|=|\ttp|+|\tte|+2\lfloor\log(|e|)\rfloor+3$$
For an optimal sharpening with a coding of pairs involving the
function
$$\log(x)+\log\log(x)+\log\log\log(x)+...$$
see Li \& Vitanyi's book \cite{livitanyi}, Example 1.11.13, p.79.
\end{remark}
\begin{proposition}[Existence of good universal functions]
\label{p:good}$\\ $
Every self-enumerated system contains a good universal function
with $c$ as associated $comp$ function.
\end{proposition}
\begin{proof}
The usual proof works.
Let $U$ and $comp_U$ be as in Def.\ref{def:self} and set
\medskip\\\medskip\centerline{$U_{opt}
=U\circ comp_U\circ (\pi_1,\pi_2)$}
Then $comp_U\circ (\pi_1,\pi_2):\words\to\words$
is total recursive and condition ii of Def.\ref{def:self} insures
that $U_{opt}\in{\cal F}$. Now, we have
\medskip\\\medskip\centerline{$
U_{opt}(c(\tte,\ttp))
= U(comp_U((\pi_1,\pi_2)(c(\tte,\ttp))))
= U(comp_U(\tte,\ttp))$}
so that $U_{opt}$ is universal with $c$ as associated
$comp$ function.
\end{proof}
%
%
\subsection{Relativization of self-enumerated representation
systems} \label{ss:relativization}
%
Def.\ref{def:self} can be obviously relativized to any oracle $A$.
However, contrary to what can be a priori expected, this is no
generalization but particularization.
The main reason is Prop.\ref{p:good}: there always exists a
universal function with $c$ as associated $comp$ function.
\begin{definition} \label{def:selfA}
Let $A\subseteq\bbbn$.
A self-enumerated representation $A$-system is a pair $(D,{\cal F})$
where ${\cal F}$ is a family of partial functions $\words\to D$
satisfying condition i of Def.\ref{def:self} and the following
variants of conditions ii and iii :
\begin{enumerate}
\item[$ii^A$.]
If $\varphi:\words\to\words$ is an $A$-recursive total function
and $F\in{\cal F}$ then $F\circ \varphi\in{\cal F}$.
\item[$iii^A$.]
There exists $U\in{\cal F}$ and a total
$A$-recursive function $comp_U:\words\times\words\to\words$
such that
$$\forall F\in{\cal F}\ \exists \tte\in\words\
\forall \ttp\in\words\
F(\ttp)=U(comp_U(\tte,\ttp))$$
\end{enumerate}
\end{definition}
\begin{example}\label{ex:selfA}
If $\bbbx$ is a basic set then $(\bbbx,\PR[A,\words\to\bbbx])$ is
obviously a self-enumerated representation $A$-system.
\end{example}

\begin{proposition} \label{p:selfA}
Every self-enumerated representation $A$-system contains a universal
function with $c$ as associated $comp$ function.
\\
In particular, every such system is also a self-enumerated
representation system.
Thus, $(\bbbx,\PR[A,\words\to\bbbx])$ is a self-enumerated
representation system.
\end{proposition}
\begin{proof}
We repeat the same easy argument used for Prop.\ref{p:good}.
Let $U$ and $comp_U$ be as in condition $iii^A$ of
Def.\ref{def:selfA} and set
$U_{opt}=U\circ comp_U\circ (\pi_1,\pi_2)$.
Then $comp_U\circ (\pi_1,\pi_2):\words\to\words$ is total
$A$-recursive and condition $ii^A$ insures that
$U_{opt}\in{\cal F}$ and we have
\medskip\\\medskip\centerline{$
U_{opt}(c(\tte,\ttp))
= U(comp_U((\pi_1,\pi_2)(c(\tte,\ttp))))
= U(comp_U(\tte,\ttp))$}
so that $U_{opt}$ is universal with $c$ as associated
$comp$ function.
\end{proof}
%
%
\subsection{The Invariance Theorem} \label{ss:invariance}
%
\begin{definition}\label{def:Kphi}
Let $F:\words\to D$ be any partial function.
The Kolmogorov complexity
$K_F^D:D\to\bbbn\cup\{+\infty\}$
associated to $F$ is the function defined as follows:
 $$K_F^D(x) = \min\{|\ttp|:F(\ttp)=x\}$$
(Convention: $\min\,\emptyset=+\infty$)
\end{definition}
\begin{remark}\label{rk:Kphi}$\\ $
{\bf 1.} $K_F^D(x)$ is finite if and only if $x\in Range(F)$.
Hence $K_F^D$ has values in $\bbbn$ (rather than
$\bbbn\cup\{+\infty\}$) if and only if $F$ is surjective.
\medskip\\
{\bf 2.} If $F:\words\to D$ is a restriction of $G:\words\to D$
then $K^D_G\leq K^D_F$.
\end{remark}
\medskip
Thanks to Prop. \ref{p:good}, the usual Invariance Theorem
can be extended to any self-enumerated representation system,
which allows to define Kolmogorov complexity for such a system.
\begin{theorem}[Invariance Theorem,
Kolmogorov, 1965 \cite{kolmo65}]\label{thm:invar}$\\ $
Let $(D,{\cal F})$ be a self-enumerated representation system.
\medskip\\
{\bf 1.}
When $F$ varies in the family ${\cal F}$,
there is a least $K_F^D$, up to an additive constant
(cf. Notation \ref{not:leqct}):
$$\exists F\in {\cal F}\ \ \forall G\in {\cal F}\
\ \ K_F^D \leqct K_G^D$$
Such $F$'s are said to optimal in ${\cal F}$.
\medskip\\
{\bf 2.}
Every good universal function for ${\cal F}$ is optimal.
\end{theorem}
\begin{proof}
It suffices to prove 2. The usual proof works.
Consider a good universal enumeration $U$ of ${\cal F}$.
Let $F\in {\cal F}$ and let $\tte$ be such that
$$U(comp_U(\tte,\ttp))=F(\ttp)\
\mbox{ for all }p\in\words$$
First, since $U$ is surjective (Prop.\ref{p:onto}),
all values of $K^D_U$ are finite.
Thus, $K^D_U(x) < K^D_F(x)$ for $x\notin Range(F)$
(since then $K^D_F(x)=+\infty$).
\\ For every $x\in Range(F)$, let $\ttp_x$ be a smallest program
such that $F(\ttp_x)=x$, i.e. $K^D_F(x)=|\ttp_x|$.
Then,
\medskip\\\centerline{$x=F(\ttp_x)=U(comp_U(\tte,\ttp_x))$}
and since $U$ is good,
\\\medskip\centerline{$K^D_U(x) \leq
|comp_U(e,\ttp_x)|\leq|\ttp_x|+c_\tte=K^D_F(x)+c_\tte$}
and therefore $K^D_U\leqct K^D_F$.
\end{proof}
As usual, Theorem \ref{thm:invar} allows for an intrinsic
definition of the Kolmogorov complexity associated to the
self-enumerated system $(D,{\cal F})$.
\begin{definition}[Kolmogorov complexity of a self-enumerated
representation system]\label{def:Kself}$\\ $
Let $(D,{\cal F})$ be a self-enumerated representation system.
\\
The Kolmogorov complexity $\ K^D_{\cal F}:D\to\bbbn\ $
is the function $K_U^D$ where $U$ is some fixed
{\em good universal enumeration} in ${\cal F}$.
\\ Up to an additive constant, this definition is independent
of the particular choice of $U$.
\end{definition}
The following straightforward result, based on Examples
\ref{ex:self} and \ref{ex:selfA}, insures that Def.\ref{def:Kself}
is compatible with the usual Kolmogorov complexity and its
relativizations.
\begin{proposition}
Let $A\subseteq\bbbn$ be an oracle and let $D=\bbbx$ be a basic set
(cf. Def.\ref{def:basic}).
The Kolmogorov complexities $K^\bbbx_{\PR[\words\to\bbbx]}$ and
$K^\bbbx_{\PR[A,\words\to\bbbx]}$ defined above are exactly
the usual Kolmogorov complexity
$K_\bbbx:\bbbx\to\bbbn$ and its relativization $K_\bbbx^A$
(cf. \S\ref{ss:K}).
\end{proposition}
%
%
\section{Some operations on self-enumerated systems}
\label{s:operations}
%
%
\subsection{The composition lemma}\label{ss:subst}
%
The following easy fact is a convenient tool to
effectivize representations (cf. \S\ref{ss:why}, \ref{ss:effRep}).
We shall also use it in \S\ref{s:Delta} to go from systems
with domain $\bbbn$ to ones with domain $\bbbz$.
\begin{lemma}[The composition lemma]\label{l:circ}$\\ $
Let $(D,{\cal F})$ be a self-enumerated representation system
and $\varphi:D\to E$ be a surjective partial function.
Set $\varphi\circ{\cal F}=\{\varphi\circ F:F\in{\cal F}\}$.
\medskip\\
{\bf 1.}
\ $(E,\varphi\circ{\cal F})$ is also a self-enumerated
representation system.
Moreover, if $U$ is universal or good universal
for ${\cal F}$ then so is $\varphi\circ U$ for
$\varphi\circ{\cal F}$.
\medskip\\
{\bf 2.}
For every $x\in E$,
$$K^E_{\varphi\circ {\cal F}}(x)
\eqct\min\{K^D_{\cal F}(y):\varphi(y)=x\}$$
In particular,
$K^E_{\varphi\circ{\cal F}}\circ \varphi\
                                \leqct\ K^D_{\cal F}$
and if $\varphi:D\to E$ is a total bijection from $D$ to $E$ then
$K^E_{\varphi\circ {\cal F}}\circ \varphi\ \eqct\ K^D_{\cal F}$.
\end{lemma}
\begin{proof}
Point 1 is straightforward. As for point 2, let $U:\words\to D$ be
some universal function  for ${\cal F}$ and observe that,
for $x\in E$,
\begin{eqnarray*}
K^E_{\varphi\circ {\cal F}}(x)
&=&\min\{|\ttp|:\ttp\mbox{ such that }\varphi(U(\ttp))=x\}\\
&=&\min\{\min\{|\ttp|:\ttp\mbox{ s.t. }U(\ttp)=y\}:
y\mbox{ s.t. }\varphi(y)=x\}\\
&=&\min\{K^D_{\cal F}(y):y\mbox{ s.t. }\varphi(y)=x\}
\end{eqnarray*}
In particular, taking $x=\varphi(z)$, we get
$K^E_{\varphi\circ {\cal F}}(\varphi(z))
                                \leqct\ K^D_{\cal F}(z)$.
\\ Finally, observe that if $\varphi$ is bijective then
$z$ is the unique $y$ such that $\varphi(y)=x$,
so that the above $\min$ reduces to $K^D_{\cal F}(z)$.
\end{proof}
%
\subsection{Product of self-enumerated representation systems}
           \label{ss:product}
%
We shall need a notion of product of self-enumerated representation
systems.
\begin{theorem}\label{thm:prod}
Let $(D_1,{\cal F}_1)$ and $(D_2,{\cal F}_2)$ be self-enumerated
representation systems
\\ We identify a pair $(F_1,F_2)\in{\cal F}_1\times{\cal F}_2$
with the function $\words\to D_1\times D_2$ which maps $\ttp$ to
$(F_1(\ttp),F_2(\ttp))$.
\\ Then $(D_1\times D_2, {\cal F}_1\times{\cal F}_2)$ is also a
self-enumerated representation system.
\\ If $(D_1,{\cal F}_1)$ and $(D_2,{\cal F}_2)$ are full systems
then so is $(D_1\times D_2, {\cal F}_1\times{\cal F}_2)$.
\\ If $U_1,U_2$ are universal for ${\cal F}_1,{\cal F}_2$ then
$$U_{1,2}
=(U_1\circ\pi_1,U_2\circ\pi_2)$$
is universal for ${\cal F}_1\times{\cal F}_2$.
\end{theorem}
\begin{proof}
{\em Condition ii} in Def.\ref{def:self} is obvious.
\medskip\\
{\em Condition i.} Let $(d_1,d_2)\in D_1\times D_2$.
Applying condition i to $(D_1,{\cal F}_1)$ and to $(D_2,{\cal F}_2)$,
we get $F_1\in{\cal F}_1$, $F_2\in{\cal F}_2$ and
$\ttp_1,\ttp_2\in\words$ such that
$d_1=F_1(\ttp_1)$ and $d_2=F_2(\ttp_2)$.
Therefore
$(d_1,d_2)=(F_1\circ\pi_1,F_2\circ\pi_2)(c(\ttp_1,\ttp_2))$.
Observe finally that
$(F_1\circ\pi_1,F_2\circ\pi_2)\in{\cal F}_1\times {\cal F}_2$
(condition ii for $(D_1,{\cal F}_1) , (D_2,{\cal F}_2)$).
\medskip\\
{\em Condition iii.} Let $comp_1,comp_2:\words\to\words$ be
the $comp$ functions associated to the universal functions
$U_1, U_2$ and set
\\\centerline{$comp_{1,2}(\tte,\ttp)
     =c(comp_1(\pi_1(\tte),\ttp),comp_2(\pi_2(\tte),\ttp))$}
For every $(F_1,F_2)\in{\cal F}_1\times {\cal F}_2$
there exist $\tta,\ttb\in\words$ such that
$F_1(\ttp)=U_1(comp_1(\tta,\ttp))$ and
$F_2(\ttp)=U_2(comp_2(\ttb,\ttp))$. Therefore
\begin{eqnarray*}
(F_1,F_2)(\ttp)
&=&(U_1(comp_1(\tta,\ttp)),U_2(comp_2\tt(\ttb,\ttp)))\\
&=&(U_1\circ\pi_1,U_2\circ\pi_2)
    (c(comp_1(\tta,\ttp),comp_2(\ttb,\ttp)))\\
&=&U_{1,2}(comp_{1,2}(c(\tta,\ttb),\ttp))
\end{eqnarray*}
which proves that $U_{1,2}$ is universal
for the product system ${\cal F}_1\times {\cal F}_2$.
\end{proof}
\begin{remark}
Observe that, even if $U_1,U_2$ are good,
the above universal function $U_{1,2}$ is not good since
\begin{eqnarray*}
|comp_{1,2}(\tte,\ttp)| & = & 2|comp_1(\pi_1(\tte),\ttp)|
                              +|comp_2(\pi_2(\tte),\ttp)|+1
\end{eqnarray*}
which is $\geq 3|\ttp|$ in general.
\\ To get a good function $\widetilde{U_{1,2}}$, argue as in
the proof of Prop.\ref{p:good}:
\begin{eqnarray*}
\widetilde{U_{1,2}}(\ttp)
&=&U_{1,2}\circ comp_{1,2}\circ(\pi_1,\pi_2)(\ttp)\\
&=&U_{1,2}(comp_{1,2}(\pi_1(\ttp),\pi_2(\ttp)))\\
&=&U_{1,2}(c(comp_1(\pi_1\pi_1(\ttp),\pi_2(\ttp)),
             comp_2(\pi_2\pi_1(\ttp),\pi_2(\ttp))))\\
&=&(U_1\circ\pi_1, U_2\circ\pi_2)
\\ &&(c(comp_1(\pi_1\pi_1(\ttp),\pi_2(\ttp)),
       comp_2(\pi_2\pi_1(\ttp),\pi_2(\ttp))))\\
&=&(U_1(comp_1(\pi_1\pi_1(\ttp),\pi_2(\ttp))),
    U_2(comp_2(\pi_2\pi_1(\ttp),\pi_2(\ttp))))
\end{eqnarray*}
\end{remark}
%
%
\section{From domain $\bbbn$ to domain $\bbbz$}
\label{s:Delta}
%
%
\subsection{The $\Delta$ operation}\label{ss:Delta}
%
Relative integers are classically introduced as equivalence classes
of pairs of natural integers of which they are the differences.
This give a simple way to go from a self-enumerated representation
system with domain $\bbbn$ to some with domain $\bbbz$.

\begin{definition}[The $\Delta$ operation]\label{def:Delta}$\\ $
Let $\diff:\bbbn^2\to\bbbz$ be the function $(m,n)\mapsto m-n$.
\\
If $(\bbbn,{\cal F})$ is a self-enumerated representation system
with domain $\bbbn$, using notations from
Lemma \ref{l:circ} and Thm.\ref{thm:prod}, we let
$(\bbbz,\Delta{\cal F})$ be the system
$$(\bbbz,  \diff\circ ({\cal F}\times{\cal F}))$$
\end{definition}
As a direct corollary of Lemma \ref{l:circ} and
Thm.\ref{thm:prod}, we have
\begin{proposition}
If $(\bbbn,{\cal F})$ is a self-enumerated representation system
(resp. full system)
with domain $\bbbn$ then so is $(\bbbz,\Delta{\cal F})$.
\end{proposition}
%
%
\subsection{$\bbbz$ systems and $\bbbn$ systems}\label{ss:Delta2}
%
The following propositions collect some easy facts about
self-enumerated systems with domain $\bbbz$ and their associated
Kolmogorov complexities.
\begin{proposition}\label{p:delta}
Let $(\bbbz,{\cal G})$ be a self-enumerated system.
\medskip\\
{\bf 1.}
Let ${\cal F}=\{G\segment G^{-1}(\bbbn):G\in{\cal G}\}$.
Then $(\bbbn,{\cal F})$ is also a self-enumerated system and
$K^\bbbn_{\cal F}=K^\bbbz_{\cal G}\segment\bbbn$.
\medskip\\
{\bf 2.}
Denote $opp:\bbbz\to\bbbz$ the function $n\mapsto-n$.
If ${\cal G}\circ opp={\cal G}$ then
$K^\bbbz_{\cal G}\eqct K^\bbbz_{\cal G}\circ opp$.
\end{proposition}
\begin{proof}
1. Conditions i-ii of Def.\ref{def:self} are obvious.
As for iii, observe that if $U\in{\cal G}$ is universal for
${\cal G}$ then $U\segment U^{-1}(\bbbn)$ is in ${\cal F}$ and is
universal for ${\cal F}$ with the same associated $comp$ function.
Now,
$K_{U\segment U^{-1}(\bbbn)}=K_U\segment\bbbn$. Whence
$K^\bbbn_{\cal F}=K^\bbbz_{\cal G}\segment\bbbn$.
\medskip\\
2. Observe that if $\varphi,F\in{\cal G}$ and
$K_\varphi\leqct K_F$ then
$K_{\varphi\circ opp}\leqct K_{F\circ opp}$.
Since ${\cal G}\circ opp={\cal G}$, we see that if
$\varphi$ is optimal then so is $\varphi\circ opp$.
Whence $K_\varphi\eqct K_{\varphi\circ opp}$, and therefore
$K^\bbbz_{\cal G}\eqct K^\bbbz_{\cal G}\circ opp$.
\end{proof}
\begin{proposition}\label{p:deltaPR}
Let $A\subseteq\bbbn$.
\medskip\\
{\bf 1.}\ \
$\PR[A,\words\to\bbbn]=\PR[A,\words\to\bbbz]\ \cap\ (\bbbn\to\bbbn)
=\{G\segment G^{-1}(\bbbn):G\in\PR[A,\words\to\bbbz]\}$.
\\
In particular, $K^{A,\bbbz}\segment\bbbn\eqct K^{A,\bbbn}$.
\medskip\\
{\bf 2.}\ \
$\PR[A,\words\to\bbbz]
=\PR[A,\words\to\bbbz]\circ opp=\Delta \PR[A,\words\to\bbbn]$.
\\
In particular, $K^{A,\bbbz}\eqct K^{A,\bbbz}\circ opp$.
\end{proposition}
%
%
\section{Self-enumerated representation systems for r.e. sets}
\label{s:RE}
%
We now come to examples of self-enumerated systems of a somewhat
different kind, which will be used in the effectivization of
set theoretical representations of integers.
%
\subsection{Acceptable enumerations}\label{ss:acceptable}
%
Let's recall the notion of acceptable enumeration of partial
recursive functions (cf. Rogers \cite {rogers} Ex. 2.10 p.41,
or Odifrreddi \cite{odifreddi}, p.215)
\begin{definition}\label{def:acceptable}
Let $\bbbx,\bbby$ be some basic sets and $A\subseteq\bbbn$.
\medskip\\
{\bf 1.}
An enumeration $(\phi^A_\tte)_{\tte\in\words}$ of partial
$A$-recursive functions $\bbbx\to\bbby$ is {\em acceptable} if
\begin{enumerate}
\item[\bf i.]
it is partial $A$-recursive as a function
$\words\times\bbbx\to\bbby$
\item[\bf ii.]
and it satisfies the parametrization (also called s-m-n)
property: for every basic set $\bbbz$, there exists a total
$A$-recursive function $s^\bbbz_\bbbx:\words\times\bbbz\to\words$
such that, for all $\tte\in\words$, $\ttz\in\bbbz$, $\ttx\in\bbbx$,
$$\phi^A_\tte(\couple\ttz\ttx)
=\phi^A_{s^\bbbz_\bbbx(\tte,\ttz)}(\ttx)$$
where $\couple\ttz\ttx$ is the image of the pair $(\ttz,\ttx)$
by some fixed total recursive bijection $\bbbz\times\bbbx\to\bbbx$.
\end{enumerate}
{\bf 3.}
An enumeration $(W^A_\tte)_{\tte\in\words}$
of $A$-recursively enumerable subsets of $\bbbx$
is {\em acceptable} if, for all $\tte\in\words$,
$W^A_\tte=domain(\phi^A_\tte)$ for some acceptable enumeration
$(\phi^A_\tte)_{\tte\in\words}$ of partial $A$-recursive functions.
\end{definition}
We shall need Rogers' theorem
(cf. Odifreddi \cite{odifreddi} p.219).
\begin{theorem}[Rogers' theorem]\label{thm:rogers}
If $(\phi^A_\tte)_{\tte\in\words}$ and
   $(\psi^A_\tte)_{\tte\in\words}$ are
two acceptable enumerations of partial $A$-recursive functions
$\bbbx\to\bbby$, then there exists some
$A$-recursive bijection $\theta:\words\to\words$ such that
$\psi^A_\tte=\phi^A_{\theta(\tte)}$ for all $\tte\in\words$.
\end{theorem}
\begin{corollary}\label{cor:rogers}
Let $(W'^A_\tte)_{\tte\in\words}$ and $(W''^A_\tte)_{\tte\in\words}$
be two acceptable enumerations of $A$-r.e. subsets of $\bbbx$.
Then there exists an $A$-recursive bijection
$\theta:\words\to\words$ such that
$W''^A_\tte=W'^A_{\theta(\tte)}$ for all $\tte\in\words$.
\end{corollary}
\begin{proof}
Apply Roger's theorem to acceptable enumerations
$(\phi^A_\tte)_{\tte\in\words},(\psi^A_\tte)_{\tte\in\words}$
of partial $A$-recursive functions such that
$W'^A_\tte=domain(\phi^A_\tte)$ and $W''^A_\tte=domain(\psi^A_\tte)$.
\end{proof}
%
%
\subsection{Self-enumerated representation systems for r.e. sets}
                               \label{ss:RE}
%
Cor.\ref{cor:rogers} allows to get a natural intrinsic notion of
``partial $A$-computable" map $\words\to RE^A(\bbbx)$.
\begin{proposition}\label{p:FRE}
Let $RE^A(\bbbx)$ be the family of $A$-recursively enumerable
subsets of $\bbbx$ and let
    $(W'^A_\tte)_{\tte\in\words}$  and
    $(W''^A_\tte)_{\tte\in\words}$ be
two acceptable enumerations of $A$-r.e. subsets of $\bbbx$.
Let $G:\words\to RE^A(\bbbx)$.
\medskip\\
{\bf 1.}
The following conditions are equivalent:
\begin{enumerate}
\item[i.]
There exists a total $A$-recursive function $f:\words\to\words$
such that $G(\ttp)=W'^A_{f(\ttp)}$ for all $\ttp\in\words$
\item[ii.]
There exists a total $A$-recursive function $f:\words\to\words$
such that $G(\ttp)=W''^A_{f(\ttp)}$ for all $\ttp\in\words$
\end{enumerate}
{\bf 2.}
The following conditions are equivalent:
\begin{enumerate}
\item[i.]
There exists a partial $A$-recursive function $f:\words\to\words$
such that, for all $\ttp\in\words$,
$G(\ttp)=
\left\{\begin{array}{ll}
W'^A_{f(\ttp)}&\mbox{if $f(\ttp)$ is defined}\\
\mbox{undefined}&\mbox{otherwise}
\end{array}\right.$
\item[ii.]
There exists a partial $A$-recursive function $f:\words\to\words$
such that, for all $\ttp\in\words$,
$G(\ttp)=
\left\{\begin{array}{ll}
W''^A_{f(\ttp)}&\mbox{if $f(\ttp)$ is defined}\\
\mbox{undefined}&\mbox{otherwise}
\end{array}\right.$
\end{enumerate}
\end{proposition}
\begin{proof}
Applying Cor.\ref{cor:rogers}, we get
    $W''^A_{f(\ttp)}=W'^A_{\theta(f(\ttp))}$
and $W'^A_{f(\ttp)}=W'^A_{\theta^{-1}(f(\ttp))}$.
To conclude, observe that $\theta\circ f$ and $\theta^{-1}\circ f$
are both total (point 1) or partial (point 2) $A$-recursive as
is $f$.
\end{proof}
We can now come to the notion of self-enumerated systems for
r.e. sets.
\begin{definition}[Self-enumerated systems for r.e. sets]
\label{def:RE}$\\ $
Let $RE^A(\bbbx)$ be the class of $A$-r.e. subsets of the basic
set $\bbbx$.
\\
Let $(W^A_\tte)_{\tte\in\words}$ be some fixed acceptable
enumeration of $A$-r.e. subsets of $\bbbx$.
Cor.\ref{cor:rogers} insures that the families defined hereafter
do not depend on the chosen acceptable enumeration.
\medskip\\
{\bf 1.}
We let ${\cal F}^{RE^A(\bbbx)}$ be the family of all {\em total}
functions $\words\to RE^A(\bbbx)$ of the form
$\ttp\mapsto W^A_{f(\ttp)}$ where $f:\words\to\words$ varies over
{\em total} $A$-recursive functions.
\medskip\\
{\bf 2.}
We let ${\cal PF}^{RE^A(\bbbx)}$ be the family of all {\em partial}
functions $\words\to RE^A(\bbbx)$ of the form
$$\ttp\mapsto \left\{\begin{array}{ll}
W^A_{f(\ttp)}&\mbox{if $f(\ttp)$ is defined}\\
\mbox{undefined}&\mbox{otherwise}
\end{array}\right.$$
where $f:\words\to\words$ varies over {\em partial} $A$-recursive
functions.
\end{definition}
The following proposition shows that, in the definition of
${\cal F}^{RE^A(\bbbx)}$,
one can either relax the total ``$A$-recursive"
condition on $f$ to ``partial $A$-recursive" with a special
convention (different from that considered in the definition of
${\cal PF}^{RE^A(\bbbx)}$) or restrict it to some particular
$A$-recursive sequence of total functions.
\begin{proposition}\label{p:RE}
For any acceptable enumeration $(W^A_\tte)_{\tte\in\words}$
of $A$-r.e. subsets of $\bbbx$ there exists
a total $A$-recursive function
$\sigma:\words\times\words\to\words$ such that,
for any total function $G:\words\to RE^A(\bbbx)$,
the following conditions are equivalent:
\begin{enumerate}
\item[a.]
$G$ is of the form $\ \ttp\mapsto W^A_{\sigma(\tte,\ttp)}$\
for some $\tte\in\words$
\item[b.]
$G\in{\cal F}^{RE^A(\bbbx)}$
\item[c.]
For all $\ttp$,
$G(\ttp)=\left\{\begin{array}{ll}
W^A_{g(\ttp)}&\mbox{if $g(\ttp)$ is defined}\\
\emptyset&\mbox{otherwise}
\end{array}\right.$.
\end{enumerate}
\end{proposition}
\begin{proof}
Since $a\Rightarrow b\Rightarrow c$ is trivial whatever be the
total recursive function $\sigma$,
it remains to define $\sigma$ such that $c\Rightarrow a$ holds.
\\
Let $(\phi^A_\tte)_{\tte\in\words}$ be an acceptable enumeration
of partial $A$-recursive functions $\bbbx\to\bbbn$
such that
$W^A_\tte=domain(\phi^A_\tte)$.
\medskip\\
Let $(\psi^A_\tte)_{\tte\in\words}$ be an enumeration
of partial $A$-recursive functions $\words\to\words$
and let $\tta$ be such that
$\phi^A_{\psi^A_\tte(\ttp)}(\ttx)=
\phi^A_\tta(\couple {(\tte,\ttp)} \ttx)$
for all $\tte,\ttp\in\words$, $\ttx\in\bbbx$.
The parameter theorem insures that there exists a total
$A$-recursive function
$s:\words\times(\words\times\words)\to\words$
such that
$$\phi^A_{\psi^A_\tte(\ttp)}(\ttx)
=\phi^A_\tta(\couple {(\tte,\ttp)} \ttx)
=\phi^A_{s(\tta,\tte,\ttp)}(\ttx)
=\phi^A_{\sigma(\tte,\ttp)}(\ttx)$$
where $\sigma(\tte,\ttp)=s(\tta,\tte,\ttp)$.
Whence the equality
$$W^A_{\psi^A_\tte(\ttp)}=W^A_{\sigma(\tte,\ttp)}$$
which is also valid when $\psi^A_\tte(\ttp)$ is undefined,
in the sense that both sets are empty.
\medskip\\
Let $G,g$ be as in c. Since $g:\words\to\words$ is $A$-recursive,
there exists $\tte$ such that
$g(\ttp)=\psi^A_\tte(\ttp)$ for any $\ttp\in\words$.
Thus,
$$W^A_{g(\ttp)}=W^A_{\psi^A_\tte(\ttp)}=W^A_{\sigma(\tte,\ttp)}$$
an equality valid also if $g(\ttp)$ is undefined,
in the sense that all sets are empty.
\\
This proves $c\Rightarrow a$.
\end{proof}
\begin{theorem}\label{thm:RE}
$(RE^A(\bbbx),{\cal F}^{RE^A(\bbbx)})$
and $(RE^A(\bbbx),{\cal PF}^{RE^A(\bbbx)})$
are self-enumerated representation systems.
\end{theorem}
\begin{proof}
Conditions $i, ii^A$ of Def.\ref{def:self}, \ref{def:selfA} are
obvious for both systems.
\\
If $U$ satisfies $iii^A$ for $\PR[A,\words\to\bbbx]$ then
$$\ttp\mapsto\left\{
\begin{array}{ll}
W^A_{U(\ttp)}&\mbox{if $U(\ttp)$ is defined}\\
\mbox{undefined}&\mbox{otherwise}
\end{array}\right.$$
satisfies $iii^A$ for ${\cal PF}^{RE^A(\bbbx)}$
with the same associated $comp$ function.
\\
Prop.\ref{p:RE} proves that the function $\ttp\mapsto W^A_\tte$
satisfies condition $iii^A$ with $\sigma$ as $comp$ function.
Thus, $(RE^A(\bbbx),{\cal F}^{RE^A(\bbbx)})$
and $(RE^A(\bbbx),{\cal PF}^{RE^A(\bbbx)})$ are self-enumerated
$A$-systems.
We conclude using Prop.\ref{p:selfA}.
\end{proof}
\begin{remark}\label{rk:RE}
It is possible to improve Prop.\ref{p:RE} so as to get $\sigma$
total recursive (rather than $A$-recursive) in condition $a$.
This will hold for particular acceptable enumerations of $A$-r.e.
sets, with the same total recursive $\sigma$ whatever be $A$.
We sketch how this can be obtained (for more details about this type
of argument, cf. our paper \cite{ferbusgrigoOrder} \S2.3, 2.4.).
\\
Using partial computable functionals $\bbbx\times P(\bbbn)\to\bbbn$,
we can view partial $A$-recursive functions as functions obtained by
freezing the second order argument in such functionals.
We can also also consider $A$-r.e. subsets of $\bbbx$
as obtained from domains of such functionals by freezing the
second order argument.
\\
When freezing the second order argument to $A\subseteq\bbbn$,
acceptable enumerations of partial computable functionals
give acceptable enumerations of partial $A$-recursive functions.
\\
In this way, consider an acceptable enumeration
$(\Phi_\tte)_{\tte\in\words}$ of partial computable functionals
$\bbbx\times P(\bbbn)\to\bbbn$
and let ${\cal W}^A_\tte=\{\ttx:(\ttx,A)\in domain(\Phi_\tte)\}$.
Arguing as in the proof of Prop.\ref{p:RE}
(with an acceptable enumeration $(\Psi_\tte)_{\tte\in\words}$
of partial computable functionals $\words\times P(\bbbn)\to\words$)
we get
$$\Phi_{\Psi_\tte(\ttp,A)}(\ttx,A)
=\Phi_\tta(\couple {(\tte,\ttp)} \ttx,A)
=\Phi_{s(\tta,\tte,\ttp)}(\ttx,A)
=\Phi_{\sigma(\tte,\ttp)}(\ttx,A)$$
where $s$ is the total recursive function involved in the parameter
property for the acceptable enumeration
$(\Phi_\tte)_{\tte\in\words}$ and
$\sigma(\tte,\ttp)=s(\tta,\tte,\ttp)$.
\\
Now, let $G\in{\cal F}^{RE^A(\bbbx)}$ and let $g:\words\to\words$ be
total $A$-recursive such that $G(\ttp)={\cal W}^A_{g(\ttp)}$.
Let $\tte\in\words$ be such that $g=\Psi(\tte,A)$.
Then
$$\Phi_{g(\ttp)}(\ttx,A)=\Phi_{\Psi_\tte(\ttp,A)}(\ttx,A)
=\Phi_{\sigma(\tte,\ttp)}(\ttx,A)\ \ \mbox{ and }\ \
G(\ttp)={\cal W}^A_{g(\ttp)}={\cal W}^A_{\sigma(\tte,\ttp)}$$
\end{remark}
%
%
\section{Infinite computations}
\label{s:InfiniteComp}
%
Chaitin, 1976 \cite{chaitin76}, and Solovay, 1977
\cite{solovay77}, considered infinite computations producing
infinite objects (namely recursively enumerable sets) so as to
define Kolmogorov complexity of such infinite objects.
\\
Following the idea of possibly infinite computations leading to
finite output (i.e. remove the sole halting condition),
Becher \& Chaitin \& Daicz, 2001 \cite{becherchaitindaicz}
introduced a variant $K^\infty$ of Kolmogorov complexity.
\\
In our paper \cite{ferbusgrigoKmaxKmin}, 2004, we introduced two
variants $\kmax[],\kmin[]$ of Kolmogorov complexity and proved
that $K^\infty=\kmax[]$.
These variants are based on two self-enumerated representation
systems, namely the classes of $\max$ and $\min$ of partial
recursive sequences of partial recursive functions.
%
%
\subsection{Self-enumerated systems of $\max$ of partial recursive
functions}
\label{ss:maxpr}
%
\begin{notation}\label{not:maxpr}
Let $A\subseteq\bbbn$.
\\
{\bf 1.}
Let $\bbbx$ be a basic set.
Extending Notation \ref{not:PR}, we denote $Rec^{A,\words\to\bbbx}$
the family of total functions $\words\to\bbbx$ which are recursive
in $A$.
\medskip\\
{\bf 2.}
Let $\bbbx$ be $\bbbn$ or $\bbbz$.
If $f:\words\times\bbbn\to\bbbx$, we denote $\max f$
the function $(\max f)(\ttp)=\max\{f(\ttp,t):t\in\bbbn\}$
(with the convention that $\max X$ is undefined if $X$ is empty
or infinite).
\\
We define the families of functions
\begin{eqnarray*}
\maxprA[\words\to\bbbx]
&=&\{\max f:f\in \PR[A,\words\times\bbbn\to\bbbx]\}
\\
\maxrA[\words\to\bbbx]&=&\{\max f:f\in Rec^{A,\words\times\bbbn\to\bbbx}\}
\end{eqnarray*}
In case $A$ is $\ \emptyset\ $, we simply write
$\maxpr[\words\to\bbbx]$ and $\maxr[\words\to\bbbx]$.
\end{notation}
\begin{proposition}\label{p:maxprmaxr}
Let $A\subseteq\bbbn$. Then
$$(\bbbn,\maxprA[\words\to\bbbn])\ \ ,\ \
(\bbbz,\maxprA[\words\to\bbbz])\ \ ,\ \
(\bbbn,\maxrA[\words\to\bbbn])$$
are self-enumerated representation systems.
\end{proposition}
\begin{proof}
First consider the no oracle case (i.e. $A=\emptyset$).
Conditions i-ii in Def.\ref{def:self} are trivial.
The classical enumeration theorem easily extends
to $\maxpr[\words\to\bbbx]$
(cf. \cite{ferbusgrigoKmaxKmin}, Thm.4.1),
proving condition iii for $(\bbbx,\maxpr[\words\to\bbbx])$
where $\bbbx$ is $\bbbn$ or $\bbbz$.
\\
It remains to show condition iii for $\maxr[\words\to\bbbn]$.
We use the following straightforward fact
(cf. \cite{ferbusgrigoKmaxKmin}, Thm.3.6):
\begin{fact}\label{fact:rec}
{\em If $f\in\PR[\words\times\bbbn\to\bbbn]$ and
$$g(\ttp,t)=\max(\{0\}\cup\{f(\ttp,i):i\leq t\ \wedge\
f(\ttp,i) \mbox{converges in at most $t$ steps}\})$$
then
$g\in Rec^{\words\times\bbbn\to\bbbn}$ and $\max g$ is an
extension of $\max f$ with value $0$ on
$domain(\max g)\setminus domain(\max f)$ (which is the set of $n$'s
such that $f(n,t)$ is defined for no $t$).}
\end{fact}
Let $U\in\maxpr[\words\to\bbbn]$ be good universal for
$\maxpr[\words\to\bbbn]$
and let $V$ be an extension of $U$ in $\maxr[\words\to\bbbn]$
given by the above fact.
If $F\in Rec^{\words\to\bbbn}$ then it is in $\PR[\words\to\bbbn]$
and there exists $\tte$ such that $F(\ttp)=U(comp_U(\tte,\ttp))$
for all $\ttp\in\words$.
Since $V$ extends $U$ and $F$ is total, we also have
$F(\ttp)=V(comp_U(\tte,\ttp))$.
Thus, $V$ is good universal for $\maxr[\words\to\bbbn]$ with the
same associated $comp$ function.
\medskip\\
Relativization to oracle $A$ proves conditions $ii^A,iii^A$,
(cf. Def.\ref{def:selfA})
for $(\bbbx,\maxprA[\words\to\bbbx])$
and $(\bbbn,\maxr[\words\to\bbbn])$.
We conclude using Prop.\ref{p:selfA}.
\end{proof}
\begin{remark}\label{rk:minPR}$\\ $
{\bf 1.}
Fact \ref{fact:rec} implies that
$\maxpr[\words\to\bbbx]$ and $\maxr[\words\to\bbbn]$ contain the
same total functions. However, considering partial functions,
the inclusion $\maxr[\words\to\bbbx]\subset\maxpr[\words\to\bbbn]$
is strict (cf. \cite{ferbusgrigoKmaxKmin} Thm.3.6, point 1).
\medskip\\
{\bf 2.}
Let $\bbbx$ be $\bbbn$ or $\bbbz$ and let
$\minprA[\words\to\bbbx],\minrA[\words\to\bbbx]$ be defined
with $\min$ instead of $\max$ as in Point 2 of the above definition
(with the same convention that $\min\emptyset$ is undefined).
Then $(\bbbx,\minprA[\words\to\bbbx])$ is also a
self-enumerated representation system.
\\
We shall not use any $\min$ based system in this paper
because they have no simple set theoretical counterparts.
\medskip\\
{\bf 3.}
None of the systems $(\bbbz,\maxrA[\words\to\bbbz])$,
$(\bbbn,\minr[\words\to\bbbn])$ and
$(\bbbz,\minr[\words\to\bbbz])$ is self-enumerated
(cf. \cite{ferbusgrigoKmaxKmin}, Thm.4.3).
\end{remark}
%
%
\subsection{Kolmogorov complexities $\kmax[], \kmax[\emptyset'],...$}
\label{ss:Kmax}

We apply Def.\ref{def:Kself} to the self-enumerated representation
systems considered in \S\ref{ss:maxpr}.
\begin{definition}[Kolmogorov complexities]
\label{def:Kmaxpr}
Let $\bbbx$ be $\bbbn$ or $\bbbz$. We denote
$\ \kmax[A,\bbbx]:\bbbx\to\bbbn\ $ the Kolmogorov
complexity of the self-enumerated representation system
$(\bbbx,\maxprA[\words\to\bbbx])$.
\medskip\\
In case $\bbbx=\bbbn$, we omit the superscript $\bbbn$.
\medskip\\
In case $\bbbx=\bbbn$ and $A$ is $\ \emptyset\ $
we simply write $\ \kmax[]$.
\end{definition}
Using Remark \ref{rk:Kphi}, point 2, and Fact \ref{fact:rec},
it is not hard to prove the following result
(cf. \cite{ferbusgrigoKmaxKmin}, Prop.6.3).
\begin{proposition}\label{p:Kinfini}
Let $A\subseteq\bbbn$.
Then $\kmax[A]$ is also the Kolmogorov complexity of the
self-enumerated system $(\bbbn,\maxrA[\words\to\bbbn])$.
I.e.
$$\ K^\bbbn_{\maxrA[\words\to\bbbn]}
=K^\bbbn_{\maxprA[\words\to\bbbn]}$$
\end{proposition}
\begin{remark}
The above proposition has no analog with $\bbbz$ since
$\maxrA[\words\to\bbbz]$ is not self-enumerated
(cf. Remark \ref{rk:minPR}, point 3).
\end{remark}
%
\subsection{$\maxr[\words\to\bbbn]$ and $\maxpr[\words\to\bbbn]$
and infinite computations}
\label{ss:InfiniteComp}
%
The following simple result gives a machine characterization
of functions in $\maxrA[\words\to\bbbn]$
(resp. $\maxprA[\words\to\bbbn]$) which will be used in the proof
of Thm.\ref{thm:index}.
\begin{definition}\label{def:Turing}
Let ${\cal M}$ be an oracle Turing machine such that
\begin{enumerate}
\item
the alphabet of the input tape is $\{0,1\}$,
plus an end-marker to delimitate the input,
\item
the output tape is write-only and has unary alphabet $\{1\}$,
\item
there is no halting state
(resp. but there are some distinguished states).
\end{enumerate}
The partial function $F^A:\words\to\bbbn$ computed by ${\cal M}$
with oracle $A$ through infinite computation
(resp. with distinguished states) is defined as follows:
$F^A(\ttp)$ is defined with value $n$ if and only if the
infinite computation (i.e. which lasts forever) of ${\cal M}$
on input $\ttp$ outputs exactly $n$ letters $1$
(resp. and at some step the current state is a distinguished one).
\end{definition}
\begin{proposition}\label{p:Turing}
Let $A\subseteq\bbbn$ be an oracle.
A function $F:\words\to\bbbn$ is in $\maxrA[\words\to\bbbn]$
(resp. $\maxprA[\words\to\bbbn]$)
if and only if there exists an oracle Turing machine ${\cal M}$
which, with oracle $A$,  computes $F$ through infinite computation
(resp. with distinguished states)
in the sense of Def.\ref{def:Turing}.
\end{proposition}
\begin{proof}
$\Leftarrow$. The function associated to an oracle Turing machine
through infinite computation (resp. with distinguished states)
is clearly $\max f$ where $f(\ttp,t)$ is the current output at step
$t$ (resp. and is undefined while the machine has not been in some
distinguished state).
\medskip\\
$\Rightarrow$. Suppose $f:\words\times\bbbn\to\bbbn$ is total
(resp. partial) $A$-recursive and set
$$X(\ttp,t)=\{f(\ttp,t'):t'<t\ \wedge\ f(\ttp,t')
\mbox{ converges in $\leq t$ steps}\})$$
Observe that $X(\ttp,0)=\emptyset$, so that the following is indeed
an $A$-recursive definition:
\begin{eqnarray*}
\widetilde{f}(\ttp,t)&=&\left
\{\begin{array}{ll}
0\mbox{ (resp. undefined)}&\mbox{if }X(\ttp,t)=\emptyset
\\
\widetilde{f}(\ttp,t-1)+1&
\mbox{if $X(\ttp,t)\neq\emptyset\ \wedge\
\widetilde{f}(\ttp,t-1)<\max X(\ttp,t)$}
\\
\widetilde{f}(\ttp,t-1)&\mbox{otherwise}
\end{array}\right.
\end{eqnarray*}
Then $\max\widetilde{f}=\max f$. Also, the unary representation of
$\widetilde{f}(\ttp,t)$ can be simply interpreted as the current
output at step $t$ of the infinite computation
(resp. with distinguished states)
of an oracle Turing machine with input $\ttp$. So that
$\max\widetilde{f}$ is the function associated to that machine.
\end{proof}
%
\subsection{$\maxpr[\words\to\bbbn]$ and the jump}
\label{ss:Kmaxversusjump}
%
The following proposition is easy.
\begin{proposition}\label{p:maxprANDjump}
Let $A\subseteq\bbbn$ and let $\bbbx$ be $\bbbn$ or $\bbbz$.
Then $$\maxprA[\words\to\bbbx]\subset\PR[A',\words\to\bbbx]$$
\end{proposition}
\begin{proof}
1. Let $f:\words\to\bbbx$ be partial $A$-recursive. A partial
$A'$-recursive definition of $(\max f)(\ttp)$ is as follows:
\begin{enumerate}
\item[i.]
First, check whether there exists $t$ such that $f(\ttp,t)$ is
defined.
\\If the check is negative then $(\max f)(\ttp)$ is undefined.
\item[ii.]
If check i is positive then start successive steps of the
following process.
\\- At step $t$, check whether $f(\ttp,t)$ is defined,
\\- if defined, compute its value,
\\- and check whether there exists $u>t$ such that $f(\ttp,u)$
is greater than the maximum value computed up to that step.
\item[iii.]
If at some step the last check in ii is negative then halt and
output the maximum value computed up to now.
\end{enumerate}
Clearly, oracle $A'$ allows for the checks in i and ii.
Also, the above process halts if and only if $f(\ttp,t)$ is
defined for some $t$ and $\{f(\ttp,t):t\in\bbbn\}$ is bounded,
i.e. if and only if $(\max f)(\ttp)$ is defined.
In that case it outputs exactly $(\max f)(\ttp)$.
\medskip\\
2. To see that the inclusion is strict, observe that the graph
of any function in $\maxprA[\words\to\bbbx]$ is
$\Sigma^{0,A}_1\wedge\Pi^{0,A}_1$ since
$$y=(\max f)(\ttp)\ \Leftrightarrow\
((\exists t\ f(\ttp,t)=y)\ \wedge\
\neg(\exists u\ \exists z>y\ f(\ttp,u)=z))$$
Whereas the graph of functions in $\PR[A',\words\to\bbbx]$ can be
$\Sigma^{0,A'}_1$ and not $\Delta^{0,A'}_1$, i.e.
$\Sigma^{0,A}_2$ and not $\Delta^{0,A}_2$.
\end{proof}
In the vein of Prop.\ref{p:maxprANDjump}, let's mention the
following result, cf. \cite{becherchaitindaicz}
(where the proof is for $K^{\infty}$,
cf. start of \S\ref{s:InfiniteComp} above)
and \cite{ferbusgrigoKmaxKmin} Prop.7.2-3 \& Cor.7.7.
\begin{proposition}\label{p:degrees}
Let $A\subseteq\bbbn$.
\medskip\\
{\bf 1.}
$K^A$ and $\kmax[A]$ are recursive in $A'$.
\medskip\\
{\bf 2.}
$K^A \supct \kmax[A] \supct K^{A'}$.
\end{proposition}
%
%
\subsection{The $\Delta$ operation on $\maxpr[\words\to\bbbn]$
and the jump}
\label{ss:DeltaMax}
%
The following variant of Prop.\ref{p:maxprANDjump} is a normal form
for partial $A'$-recursive $\bbbz$-valued functions.
We shall use it in \S\ref{s:card}-\ref{s:index}.
\begin{theorem}\label{thm:Deltamax}
Let $A\subseteq\bbbn$. Then
$$\PR[A',\words\to\bbbz]=\Delta(\maxprA[\words\to\bbbn])
=\Delta(\maxrA[\words\to\bbbn])$$
Thus, every partial $A'$-recursive function is the difference of
two functions in $\maxrA[]$ (cf. Notation \ref{not:maxpr}).
\end{theorem}
Before entering the proof of Thm.\ref{thm:Deltamax}, let's recall
two well-known facts about oracular computation and approximation
of the jump.
\begin{lemma}\label{l:oracleCV}
Let $(B_t)_{t\in\bbbn}$ be a sequence of subsets of $\bbbn$ which
converges pointwise to $B\subseteq\bbbn$, i.e.
$$\forall n\ \ \exists t_n\ \ \forall t\geq t_n\ \ \
B_t\cap\{0,1,...,n\}=B\cap\{0,1,...,n\}$$
Let $\bbbx,\bbby$ be basic sets and let $\psi:\bbbx\to\bbby$ be a
partial $B$-recursive function computed by some oracle Turing
machine ${\cal M}$ with oracle $B$. Let $\ttx\in\bbbx$.
\\
Then, $\psi(\ttx)$ is defined if and only if there exists $t_\ttx$
such that
\begin{enumerate}
\item[i.]
the computation of ${\cal M}$ on input $\ttx$ with oracle
$B_{t_\ttx}$ halts in at most $t_\ttx$ steps,
\item[ii.]
for all $t\geq t_\ttx$ the computation of ${\cal M}$ on input $\ttx$
with oracle $B_t$ is step by step exactly the same as that with
oracle $B_{t_\ttx}$ (in particular, it asks the same questions to
the oracle, gets the same answers and halts at the same computation
step $\leq t_\ttx$).
\end{enumerate}
\end{lemma}
\begin{lemma}\label{l:approx}
Let $A\subseteq\bbbn$ and let $A'\subseteq\bbbn$ be the jump of $A$.
There exists a total $A$-recursive sequence
$(Approx(A',t))_{t\in\bbbn}$
of subsets of $\bbbn$ which is monotone increasing with respect to
set inclusion and which has union $A'$.
In particular, this sequence converges pointwise to $A'$.
\end{lemma}
We can now prove Thm.\ref{thm:Deltamax}.
\medskip\\
{\em Proof of Thm.\ref{thm:Deltamax}.}\\
Using Prop.\ref{p:maxprANDjump} and Prop.\ref{p:deltaPR}, we get
$$\Delta(\maxrA[\words\to\bbbn])
\subseteq\Delta(\maxprA[\words\to\bbbn])
\subseteq\Delta(\PR[A',\words\to\bbbn])=\PR[A',\words\to\bbbz]$$
Since $\maxrA[\words\to\bbbn]$ is closed by sums, we have
$\Delta(\Delta(\maxrA[\words\to\bbbn])
=\Delta(\maxrA[\words\to\bbbn])$.
Thus, to get the wanted equality, it suffices to prove inclusion
$$\PR[A',\words\to\bbbn]
\subseteq\Delta(\maxrA[\words\to\bbbn])$$
Let ${\cal M}$ be an oracle Turing machine with inputs in $\words$,
which, with oracle $A'$, computes the partial $A'$-recursive function
$\varphi^{A'}:\words\to\bbbn$.
\\
To prove that $\varphi^{A'}$ is in $\Delta(\maxrA[\words\to\bbbn])$,
we define total $A$-recursive functions
$f,g:\words\times\bbbn\to\bbbn$ which are (non strictly)
monotone increasing and such that $\varphi^{A'}=\max f-\max g$.
\medskip\\
The idea to get $f,g$ is as follows.
We consider $A$-recursive approximations of oracle $A'$ (as given
by Lemma \ref{l:approx}) and use them as fake oracles.
Function $f$ is obtained by letting ${\cal M}$ run with the fake
oracles and restart its computation each time some better
approximation of $A'$ shows the previous fake oracle has given
an incorrect answer.
Function $g$ collects all the outputs of the computations which
have been recognized as incorrect in the computing process for $f$.
\medskip\\
We now formally define $f,g$.
\\
First, since we do not care about computation time and space,
we can suppose without loss of generality, that, at any step $t$,
${\cal M}$ asks to the oracle about the integer $t$ and writes down
the oracle answer on the $t$-th cell of some dedicated tape.
\\
Consider $t+1$ steps of the computation of ${\cal M}$ on input
$\ttp$ with oracle $Approx(A',t)$ (cf. Lemma \ref{l:approx}).
We denote ${\cal C}_{\ttp,t+1}$ this limited computation.
We say that ${\cal C}_{\ttp,t+1}$ halts if ${\cal M}$ (with that
fake oracle) halts in at most $t+1$ steps.
\\
We denote $output({\cal C}_{\ttp,t})$ the current value (which is
in $\bbbz$) of the output tape after step $t$.
The $A$-recursive definition of $f,g$ is as follows.
\begin{enumerate}
\item[i.]
$f(\ttp,0)=g(\ttp,t)=0$
\item[ii.]
Suppose
$Approx(A',t+1)\cap\{0,...,t\}= Approx(A',t)\cap\{0,...,t\}$.
Then, up to the halting step of ${\cal C}_{\ttp,t}$ or up to
step $t$ in case ${\cal C}_{\ttp,t}$ does not halt,
both computations ${\cal C}_{\ttp,t},{\cal C}_{\ttp,t+1}$ are
stepwise identical.
\begin{enumerate}
\item
If ${\cal C}_{\ttp,t}$ halts then so does ${\cal C}_{\ttp,t+1}$
at the same step. And both computations have the same output.
\\In that case, we set
$f(\ttp,t+1)=f(\ttp,t)\ ,\ g(\ttp,t+1)=g(\ttp,t)$.
\item
If ${\cal C}_{\ttp,t}$ does not halt then let
$\delta_{t+1}=output({\cal C}_{\ttp,t+1})-output({\cal C}_{\ttp,t})$,
and set
\medskip\\
$\begin{array}{rcl}
f(\ttp,t+1)&=&f(\ttp,t)+1+\max(0,\delta_{t+1})\\
g(\ttp,t+1)&=&g(\ttp,t)+1+\max(0,-\delta_{t+1})
\end{array}$
\medskip\\
i.e. we add $|\delta_{t+1}|$ to $f$ or $g$ according to the sign
of $\delta_{t+1}$.
\end{enumerate}
\item[iii.]
Suppose
$Approx(A',t+1)\cap\{0,...,t\}\neq Approx(A',t)\cap\{0,...,t\}$.
Since these approximations are monotone increasing, we necessarily
have $Approx(A',t)\cap\{0,...,t\}\neq A'\cap\{0,...,t+1\}$.
\\
Thus, the fake oracle in ${\cal C}_{\ttp,t}$ has given answers
which are not compatible with $A'$. In that case, we set
\medskip\\
$\begin{array}{rcl}
f(\ttp,t+1)&=&
f(\ttp,t)+g(\ttp,t)+1+\max(0,output({\cal C}_{\ttp,t+1}))
\\
g(\ttp,t+1)&=&
f(\ttp,t)+g(\ttp,t)+1+\max(0,-output({\cal C}_{\ttp,t+1}))
\end{array}$
\medskip\\
i.e. we uprise $f,g$ to a common value
(namely $f(\ttp,t)+g(\ttp,t)$) and then add
$|output({\cal C}_{\ttp,t+1})|$ to $f$ or $g$ according to the sign
of $output({\cal C}_{\ttp,t+1})$.
\end{enumerate}
From the above inductive definition, we see that,
for each $t>0$,
$$f(\ttp,t)-g(\ttp,t)=output({\cal C}_{\ttp,t})$$
{\em Suppose $\varphi^{A'}(\ttp)$ is defined.}\\
Applying Lemmas \ref{l:oracleCV}, \ref{l:approx}, we see that
there exist $s_\ttp\leq t_\ttp$ such that
\\- ${\cal M}$, on input $\ttp$, with oracle $A'$, halts in
$s_\ttp$ steps,
\\-
$Approx(A',t_\ttp)\cap\{0,...,t_\ttp\}=A'\cap\{0,...,t_\ttp\}$.
\\
Thus, for all $t\geq t_\ttp$, $f_{\ttp,t}=f_{\ttp,t_\ttp}$
and $g_{\ttp,t}=g_{\ttp,t_\ttp}$ and
$f_{\ttp,t}-g_{\ttp,t}=\varphi^{A'}(\ttp)$.
\medskip\\
{\em Suppose $\varphi^{A'}(\ttp)$ is not defined.}\\
Observe that, each time the ``fake" computation ${\cal C}_{\ttp,t}$
with oracle $Approx(A',t)$ does not halt or appears not to be the
``right" one with oracle $A'$
(because $Approx(A',t+1)\cap\{0,...,t\}$ differs from
$Approx(A',t)\cap\{0,...,t\}$),
we strictly increase both $f,g$
(this is why we put $+1$ in the equations of iib and iii).
\\
Applying Lemmas \ref{l:oracleCV}, \ref{l:approx}, we see that,
if $\varphi^{A'}(\ttp)$ is not defined then ${\cal C}_{\ttp,t}$
does not halt for infinitely many $t$'s, so that
$f(\ttp,t)$ and $g(\ttp,t)$ increase infinitely often.
Therefore, $(\max f)(\ttp)$ and $(\max g)(\ttp)$ are both undefined,
and so is their difference.
\medskip\\
This proves that $\varphi^{A'}=\max f-\max g$.
Since the sequence $(Approx(A',t))_{t\in\bbbn}$ is $A$-recursive,
so are $f,g$. Thus, $\max f,\max g$ are in $\maxrA[\words\to\bbbn]$
and their difference $\varphi^{A'}$ is in
$\Delta(\maxrA[\words\to\bbbn])$.
\hfill{$\Box$}
%
%
\section{Abstract representations and effectivizations}
\label{s:abstract}
%
\subsection{Some arithmetical representations of $\bbbn$}
\label{ss:arithm}
%
As pointed in \S\ref{ss:mainResults}, abstract entities such as
numbers can be represented in many different ways.
In fact, each representation illuminates some particular role
and/or property, i.e. some possible semantics chosen in order to
efficiently access special operations or stress special properties
of integers.
\medskip\\
Usual arithmetical representations of $\bbbn$ using words on a
digit alphabet can be looked at as a (total) surjective
(non necessarily injective) function $R:C\rightarrow\bbbn$ where
$C$ is some simple free algebra or a quotient of some free algebra.
\\
Such representations are the ``degree zero" of abstraction
for representations and, as expected, their associated Kolmogorov
complexities all coincide (cf. Thm.\ref{thm:recrep} below).
\begin{example}[Base $k$ representations]\label{ex:base}$\\ $
{\bf 1.}
Integers in unary representation correspond to elements of the
free algebra built up from one generator and one unary function,
namely $0$ and the successor function $x\mapsto x+1$.
The associated function $R:{1}^*\to\bbbn$ is simply the length
function.
\medskip\\
{\bf 2.}
The various base $k$ (with $k\geq 2$) representations of
integers also involve term algebras, not necessarily free.
They differ by the set $A\subset\bbbn$ of digits they use but all
are based on the usual interpretation $R:A^*\to\bbbn$ such that
     $R(a_n \ldots a_1 a_0)=\sum_{i=0,\ldots,n} a_i k^i$.
Which, written \`a la H\"orner,\\
$$k(k(\ldots k (k a_n + a_{n-1}) + a_{n-2})\ldots) + a_1) + a_0$$
is a composition of applications
    $S_{a_0} \circ S_{a_1} \circ \ldots \circ S_{a_n}(0)$
where $S_a : x \mapsto kx+a$.
If a representation uses digits $a\in A$ then it corresponds to
the algebra generated by $0$ and the $S_a$'s where $a\in A$.
\begin{enumerate}
\item[i.] The $k$-adic representation uses digits $1,2,\ldots,k$
      and corresponds to a free algebra built up from one
      generator and $k$ unary functions.
\item[ii.] The usual $k$-ary representation uses digits
      $0,1,\ldots,k-1$ and corresponds to the quotient
      of a free algebra built up from one generator and
      $k$ unary functions,
      namely $0$ and the $S_a$'s where $a=0,2,\ldots, k-1$,
      by the relation $S_0(0)=0$.
\item[iii.] Avizienis base $k$ representation uses digits
      $-k+1,\ldots,-1,0,1,\ldots,k-1$
      (it is a much redundant representation used in computers
      to perform additions without carry propagation) and
      corresponds to the quotient of the free algebra built up from
      one generator and $2k-1$ unary functions,
      namely $0$ and the $S_a$'s where
               $a=-k+1,\ldots,-1,0,1,\ldots, k-1$,
      by the relations
      $\ \forall x\ (S_{-k+i}\circ S_{j+1}(x)=S_i\circ S_j(x))\ $
      where $-k<j<k-1$ and $0<i<k$.
\end{enumerate}
\end{example}
\noindent
Somewhat exotic representations of integers can also be associated
to deep results in number theory.
\begin{example}\label{ex:exotic}$\\ $
{\bf 1.}
$R: \bbbn^4\to \bbbn$ such that $R(x,y,z,t)=x^2+y^2+z^2+t^2$
is a representation based on Lagrange's four squares
theorem.  
\medskip\\
{\bf 2.}
$R:(Prime\cup\{0\})^7\to \bbbn$ such that
         $R(x_1,\ldots ,x_i)=x_1+\ldots+x_i$
is a representation based on Schnirelman's theorem (1931) in its
last improved version obtained by Ramar\'e, 1995 \cite{ramare},
which insures that every even number is the sum of at most 6 prime
numbers (hence every number is the sum of at most 7 primes).
\end{example}
\noindent
Such representations appear in the study of the expressional
power of some weak arithmetics.
For instance, the representation as sums of 7 primes allows for
a very simple proof of the definability of multiplication with
addition and the divisibility predicate
(a result valid in fact with successor and divisibility,
(Julia Robinson, 1948 \cite{robinson})).
%
%
\subsection{Abstract representations}     \label{ss:abstract}
%
Foundational questions, going back to Russell, \cite{russell08} 1908,
and Church, \cite{church33} 1933, lead to
quite different representations of $\bbbn$ : set theoretical
representations involving abstract sets and functionals much more
complex than the integers they represent.
\medskip\\
We shall consider the following simple and general notion.
\begin{definition}[Abstract representations]\label{def:rep}$\\ $
A representation of an infinite set $E$ is a pair $(C,R)$
where $C$ is some (necessarily infinite) set and
$R : C \rightarrow E$ is a {\em surjective partial} function.
\end{definition}
\begin{remark}$\\ $
{\bf 1.}
Though $R$ really operates on the sole subset $domain(R)$,
the underlying set $C$ is quite significant in the
effectivization process which is necessary to get a self-enumerated
systen and then an associated Kolmogorov complexity.
\medskip\\
{\bf 2.}
We shall consider representations with arbitrarily complex
domains in the Post hierarchy
(cf. Prop.\ref{p:complexEffCard}, \ref{p:complexEffIndex},
\ref{p:complexEffChurch}, and coming papers).
In fact, the sole cases in this paper where $R$
is a total function are the usual recursive representations.
\medskip\\
{\bf 3.}
Representations can also involve a proper class $C$
(cf. Rk. \ref{rk:cardRep}).
However, we shall stick to the case $C$ is a set.
\end{remark}
%
\subsection{Effectivizing representations: why?} \label{ss:why}
%
Turning to a computer science (or recursion theoretic) point of
view, there are some objections to the consideration of abstract
sets, functions and functionals as we did in \S\ref{ss:mainResults}
and \ref{ss:abstract}:
\begin{itemize}
\item   We cannot apprehend abstract sets, functions and
        functionals but solely programs to compute them
        (if they are computable in some sense).
\item  Moreover, programs dealing with sets, functions and
       functionals have to go through some intensional
       representation of these objects in order to be able to
       compute with such objects.
\end{itemize}
To get effectiveness, we turn from set theory to computability
theory. We shall do that in a somewhat abstract way using
self-enumerated representation systems (cf. Def.\ref{def:self}).
\\
We shall consider higher order representations and shall
``effectivize" abstract sets, functions and functionals via
recursively enumerable sets, partial recursive functions or
$\max$ of total or partial recursive functions,
and partial computable functionals.
%
\subsection{Effectivizations of representations and associated
            Kolmogorov complexities}        \label{ss:effRep}
%
A formal representation of an integer $n$ is a finite object
(in general a word) which describes some characteristic property
of $n$ or of some abstract object which characterizes $n$.
To effectivize a representation $\ R:C\to E\ $, we shall process
as follows:
\begin{enumerate}
\item
Restrict the set $C$ to a subfamily $D$ of elements which,
in some sense, are computable or partial computable.
Of course, we want the restriction of $R$ to $D$ to be still
surjective.
\item
Consider a self-enumerated representation system for $D$.
\end{enumerate}
This leads to the following definition.
\begin{definition}\label{def:effRep}$\\ $
{\bf 1.}
A set $D$ is adapted to the representation
$\ R:C\to E\ $ if $D\subseteq C$ and the partial function
$\ R\segment D:D\to E\ $ is still surjective.
\medskip\\
{\bf 2.}
{\bf [Effectivization]}
An effectivization of the representation $\ R:C\to E\ $ of the set
$E$ is any self-enumerated representation system $(D,{\cal F})$ for
a domain $D$ adapted to the representation $\ R:C\to E\ $.
\end{definition}
Using the Composition Lemma \ref{l:circ}, we immediately get
\begin{proposition}
Let $\ R:C\to E$ be a representation of $E$ and $(D,{\cal F})$ be
some effectivization of $R$.
Then $(E,(R\segment D)\circ{\cal F})$ is a self-enumerated
representation system and the associated Kolmogorov complexity
$K^E_{(R\segment D)\circ{\cal F}}$ (cf. Def.\ref{def:Kself})
satisfies
$$K^E_{(R\segment D)\circ{\cal F}}(x)
=\min\{K^D_{\cal F}(y):R(y)=x\}\ \mbox{ for all }x\in E$$
\end{proposition}
\begin{remark}
Whereas abstract representations are quite natural and
conceptually simple, the functions $\ (R\segment D)\circ F\ $,
for $F\in{\cal F}$, in the self-enumerated representation families
of their effectivized versions may be quite complex.
In the examples we shall consider, their domains involve levels
$2$ or $3$ of the arithmetical hierarchy.
In particular, such representations are not Turing
reducible one to the other.
\end{remark}
%
%
\subsection{Partial recursive representations}\label{ss:recrep}
%
We already mentioned in \S\ref{ss:arithm} that all usual arithmetic
representations lead to the same Kolmogorov complexity (up to an
additive constant).
The following result extends this assertion to all partial recursive
representations.
\begin{theorem}\label{thm:recrep}
We keep the notations of Notations \ref{not:PR} and
Def.\ref{def:Kself}.
\\
Let $A\subseteq\bbbn$ be an oracle.
If $C,E$ are basic sets and $R:C\to E$ is partial recursive
(resp. partial $A$-recursive) then
\medskip\\
$\begin{array}{rcllrcll}
R\circ \PR[\words\to C]&=&\PR[\words\to E]
&\mbox{(resp. }
&R\circ \PR[A,\words\to C]&=&\PR[A,\words\to E]
&\mbox{)}
\medskip\\
K^E_{R\circ \PR[\words\to C]}&=&K_E
&\mbox{(resp. }
&K^E_{R\circ \PR[A,\words\to C]}&=&K^A_E
&\mbox{)}
\end{array}$
\medskip\\
Thus, all Kolmogorov complexities associated to partial recursive
(resp. partial $A$-recursive) representations of $E$ coincide with
the usual (resp. $A$-oracular) Kolmogorov complexity on $E$.
\end{theorem}
\begin{proof}
It suffices to prove that
$$R\circ \PR[A,\words\to C] = \PR[A,\words\to E]$$
Inclusion
$R\circ \PR[A,\words\to C] \subseteq \PR[A,\words\to E]$
is trivial.
For the other inclusion, we use the fact that $R:C\to E$ is
surjective partial $A$-recursive.
\\
First, define a partial $A$-recursive $S:E\to C$ such that,
for $x\in E$, $S(\ttx)$ is the element $\tty\in C$ satisfying
$R(\tty)=\ttx$ which appears first in an $A$-recursive enumeration
of the graph of $R$.
Clearly, $S$ is a right inverse of $R$,
i.e. $R\circ S=Id_E$ where $Id_E$ is the identity on $E$.
\\
Using the trivial inclusion
$S\circ \PR[A,\words\to E]\subseteq \PR[A,\words\to C]$
we get
$$\PR[A,\words\to E]
=R\circ S\circ \PR[A,\words\to E]
\subseteq R\circ \PR[A,\words\to C]$$
\end{proof}
%
%
%
\section{Cardinal representations of $\bbbn$}
\label{s:card}
%
%
\subsection{Basic cardinal representation and its effectivizations}
\label{ss:effCard}
%
Among the conceptual representations of integers, the most basic
one goes back to Russell,
\cite{russell08} 1908 (cf. \cite{heijenoort} p.178), and considers
non negative integers as equivalence classes of sets relative
to cardinal comparison.
\begin{definition}[Cardinal representation of $\bbbn$]
\label{def:card}
Let $\card(Y)$ denote the cardinal of $Y$, i.e. the number
of its elements.
\\
The cardinal representation of $\bbbn$ relative to an infinite
set $X$ is the partial function $$\card_X:P(X)\to\bbbn$$
with domain $P^{<\omega}(X)$, such that
$$\card_X(Y)=\left\{
\begin{array}{ll}
         \card(Y)&\mbox{if $Y$ is finite}\\
\mbox{undefined}&\mbox{otherwise}
\end{array}\right.$$
\end{definition}
\begin{definition}[Effectivizations of the cardinal representation
of $\bbbn$]\label{def:effCard}
We effectivize the cardinal representation by replacing
$P(X)$ by $RE(\bbbx)$ or $RE^A(\bbbx)$ where $\bbbx$ is some
basic set and $A\subseteq\bbbn$ is some oracle.
\\
Two kinds of self-enumerated representation systems can be
naturally associated to these domains
(cf. \S\ref{ss:RE} and the Composition Lemma \ref{l:circ}):%
\begin{eqnarray*}
(RE(\bbbx),\card\circ{\cal F}^{RE(\bbbx)})&\mbox{or}&
(RE^A(\bbbx),\card\circ{\cal F}^{RE^A(\bbbx)})
\\
(RE(\bbbx),\card\circ{\cal PF}^{RE(\bbbx)})&\mbox{or}&
(RE^A(\bbbx),\card\circ{\cal PF}^{RE^A(\bbbx)})
\end{eqnarray*}
\end{definition}
\begin{remark}\label{rk:cardRep}$\\ $
{\bf 1.}
Historically, the cardinal representation of $\bbbn$ considered
the whole class of sets rather than some $P(X)$.
However, the above effectivization makes such an extension
unsignificant for our study.
\medskip\\
{\bf 2.}
One can also consider the total representation obtained by
restriction to the set $P_{<\omega}(X)$ of all finite subsets
of $X$. But this amounts to a partial recursive representation
and is relevant to \S\ref{ss:recrep}.
\end{remark}
%
%
\subsection{Syntactical complexity of cardinal representations}
\label{ss:syntaxcard}
%
%
The following proposition gives the syntactical complexity of
the above effectivizations of the cardinal representations.
\begin{proposition}[Syntactical complexity]
\label{p:complexEffCard}
The family
$$\{domain(\varphi):\varphi\in \card\circ{\cal F}^{RE^A(\bbbx)}\}$$
is exactly the family of $\Sigma^{0,A}_2$ subsets of $\words$.
Idem with $\card\circ{\cal PF}^{RE^A(\bbbx)}$.
\medskip\\
In particular, any universal function for
$\card\circ{\cal F}^{RE^A(\bbbx)}$ or for
$\card\circ{\cal PF}^{RE^A(\bbbx)}$
is $\Sigma^{0,A}_2$-complete.
\end{proposition}
\begin{proof}
Let $(W^A_\tte)_{\tte\in\words}$ be an acceptable enumeration of
$RE^A(\bbbx)$.
\\
1. If $g:\words\to\words$ is partial $A$-recursive then
$$domain(\ttp\mapsto\card(W^A_{g(\ttp)})
=\{\ttp:W^A_{g(\ttp)}\mbox{ is finite}\}$$
is clearly $\Sigma^{0,A}_2$.
\medskip\\
2. Let $X\subseteq\words$ be a $\Sigma^{0,A}_2$ set of the form
$X=\{\ttp:\exists u\ \forall v\ R(\ttp,u,v)\}$ where
$R\subseteq\words\times\bbbn^2$ is $A$-recursive.
Set
\begin{eqnarray*}
\sigma_\ttp&=&\left\{\begin{array}{ll}
\{u':u'<u\}&\mbox{if $u$ is least such that }\forall v\ R(\ttp,u,v)
\\
\bbbn&\mbox{if there is no $u$ such that $\forall v\ R(\ttp,u,v)$}
\end{array}\right.
\end{eqnarray*}
It is easy to check that $\sigma_\ttp\subseteq\bbbn$ is an $A$-r.e. set
which can be defined by the following enumeration process described in
Pascal-like instructions:
\medskip\\\centerline{\tt\begin{tabular}{lll}
\{Initialization\}&$u:=0$; $v:=0$;\\
\{Loop\}&DO FOREVER&BEGIN\\
&&WHILE $R(\ttp,u,v)$ DO $v:=v+1$;\\
&&output $u$ in $\sigma_\ttp$;\\
&&$u:=u+1$; $v:=0$;\\
&&END;
\end{tabular}}
\\
Clearly, $card(\sigma_\ttp)$ is finite if and only if $\ttp\in X$.
\medskip\\
Now, the set $\{(\ttp,n):n\in\sigma_\ttp\}$ is also $A$-r.e.,
hence of the form $W_\tta^{\words\times\bbbn}$ for some $\tta$.
The parameter property yields a total $A$-recursive function
$g:\words\to\words$ such that $\sigma_\ttp=W_{g(\tta,\ttp)}$.
Finally, the function $\ttp\mapsto\card(W_{g(\tta,\ttp)})$ is in
$\card\circ{\cal F}^{RE^A(\bbbx)}$ and has domain $X$.
\end{proof}
%
%
\subsection{Characterization of the $card$ self-enumerated systems}
\label{ss:characterizeCard}

\begin{theorem}\label{thm:card}
For any basic set $\bbbx$ and any oracle $A\subseteq\bbbn$,
\medskip\\
$\begin{array}{lrcl}
\mbox{\bf 1i.}&
\card\circ{\cal F}^{RE^A(\bbbx)}&=&\maxrA[\words\to\bbbn]
\medskip\\
{\bf \ ii.}&\card\circ{\cal PF}^{RE^A(\bbbx)}
&=&\maxprA[\words\to\bbbn]
\end{array}$
\medskip\\
$\begin{array}{lrclcl}
\mbox{\bf 2.}&
K^\bbbn_{\card\circ{\cal F}^{RE^A(\bbbx)}}&\eqct&
K^\bbbn_{\card\circ{\cal PF}^{RE^A(\bbbx)}}&\eqct&\kmax[A]
\end{array}$
\medskip\\
We shall simply write $K^{\bbbn,A}_{\card}$
in place of $K^\bbbn_{\card\circ{\cal F}^{RE^A(\bbbn)}}$.
\\
When $A=\emptyset$ we simply write $K^{\bbbn}_{\card}$.
\end{theorem}
\begin{proof}
Point 2 is a direct corollary of Point 1 and Prop.\ref{p:Kinfini}.
Let's prove point 1.
\medskip\\
{\bf 1i. }{\em Inclusion $\subseteq$.}\\
Let $g:\words\to\words$ be total $A$-recursive.
We define a total $A$-recursive function
$u:\words\times\bbbn\to\bbbn$ such that
$$(*)\ \ \ \{u(\ttp,t):t\in\bbbn\}=
\left\{\begin{array}{ll}
\{0,...,n\}&\mbox{if $W^A_{g(\ttp)}$ contains exactly $n$ points}
\\
\bbbn&\mbox{if $W^A_{g(\ttp)}$ is infinite}
\end{array}\right.$$
\noindent
The definition is as follows.
First, set $u(\ttp,0)=0$ for all $\ttp$.
Consider an $A$-recursive enumeration of $W^A_{g(\ttp)}$.
If at step $t$, some new point is enumerated then set
$u(\ttp,t+1)=u(\ttp,t)+1$, else set $u(\ttp,t+1)=u(\ttp,t)$.
\medskip\\
From $(*)$ we get $\card(W_\ttp)=(\max f)(\ttp)$, so that
$\ttp\mapsto\card(W^A_{g(\ttp)})$ is in
$\maxrA[\words\to\bbbn]$.
\medskip\\
{\bf 1ii. }{\em Inclusion $\subseteq$.}\\
Now $g:\words\to\words$ is partial $A$-recursive and we define
$u:\words\times\bbbn\to\bbbn$ as a partial $A$-recursive function
such that
$$\{u(\ttp,t):t\in\bbbn\}=
\left\{\begin{array}{ll}
\emptyset&\mbox{if $g(\ttp)$ is undefined}
\\
\{0,...,n\}&\mbox{if $W^A_{g(\ttp)}$ contains exactly $n$ points}
\\
\bbbn&\mbox{if $W^A_{g(\ttp)}$ is infinite}
\end{array}\right.$$
\noindent
The definition of $u$ is as above except that, for any $t$,
we require that $u(\ttp,t)$ is defined if and only if $g(\ttp)$ is.
\medskip\\
{\bf 1i. }{\em Inclusion $\supseteq$.}\\
Any function in $\maxrA[\words\to\bbbn]$ is of the form
$\max f:\words\to\bbbn$ where $f:\words\times\bbbn\to\bbbn$ is
total $A$-recursive.
\\The idea to prove that $\max f$ is in
$\card\circ{\cal F}^{RE^A(\bbbx)}$ is quite simple.
For every $\ttp$, we define an $A$-r.e. subset of $\bbbx$ which
collects some new elements each time $f(\ttp,t)$ gets greater than
$\max\{f(\ttp,t'):t'<t\}$.
\\
Formally, let $\psi:\words\times\bbbn\to\bbbn$ be the partial
$A$-recursive function such that
$$\psi(\ttp,t)=
\left\{\begin{array}{ll}
0&\mbox{if $\exists u\ f(\ttp,u)>t$}\\
\mbox{undefined}&\mbox{otherwise}
\end{array}\right.$$
Clearly,
$$domain(\psi_\ttp)
=\left\{\begin{array}{ll}
\{t:0\leq t<(\max f)(\ttp)\}&\mbox{if $(\max f)(\ttp)$ is defined}\\
\bbbn&\mbox{otherwise}
\end{array}\right.$$
We define $\varphi:\words\times\bbbx\to\bbbn$ such that
$\varphi(\ttp,\ttx)=\psi(\ttp,\theta(\ttx))$ where
$\theta:\bbbx\to\bbbn$ is some fixed total recursive bijection.
Let's denote $\psi_\ttp$ and $\varphi_\ttp$ the functions
$t\mapsto\psi(\ttp,t)$ and $\ttx\mapsto\varphi(\ttp,\ttx)$.
Let $\tte$ be such that
$W^A_\tte=\{\couple\ttp\ttx:(\ttp,\ttx)\in domain(\varphi)\}$
(where $\couple\,\,$ is a bijection $\words\times\bbbx\to\bbbx$).
The parameter property yields an $A$-recursive function
$s:\words\times\words\to\words$ such that
$W^A_{s(\tte,\ttp)}=domain(\varphi_\ttp)$ for all $\ttp$.
Thus, letting $g:\words\to\words$ be the $A$-recursive function
such that $g(\ttp)=s(\tte,\ttp)$, we have
$$\card(W^A_{g(\ttp)})
=\card(domain(\varphi_\ttp))
=\card(domain(\psi_\ttp))=(\max f)(\ttp)$$
Which proves that $\max f$ is in
$\card\circ{\cal F}^{RE^A(\bbbx)}$.
\medskip\\
{\bf 1ii. }{\em Inclusion $\supseteq$.}\\
We argue as in the above proof of {\bf i. $\supseteq$.}
However, $f:\words\times\bbbn\to\bbbn$ is now partial $A$-recursive
and there are two reasons for which $(\max f)(\ttp)$ may be
undefined: first, if $t\mapsto f(\ttp,t)$ is unbounded, second if it
has empty domain.
Keeping $\psi$ and $\varphi$ as defined as above, we now have,
$$domain(\psi_\ttp)
=\left\{\begin{array}{ll}
\{v:0\leq v<(\max f)(\ttp)\}&\mbox{if $(\max f)(\ttp)$ is defined}\\
\bbbn&\mbox{if $range(t\mapsto f(\ttp,t))$ is infinite}\\
\emptyset&\mbox{if $f(\ttp,t)$ is defined for no $t$}
\end{array}\right.$$
We let $\tte,s,g$ be as above and define $h:\words\to\words$
such that
$$h(\ttp)
=\left\{\begin{array}{ll}
g(\ttp)&\mbox{if $f(\ttp,t)$ is defined for some $t$}\\
\mbox{undefined}&\mbox{otherwise}
\end{array}\right.$$
Observe that
\\\indent- if $t\mapsto f(\ttp,t)$ has empty domain then $h(\ttp)$
is undefined,
\\\indent- if $t\mapsto f(\ttp,t)$ is unbounded then
$card(W^A_{h(\ttp)})=card(W^A_{g(\ttp)})$ is infinite,
\\\indent- otherwise
$card(W^A_{h(\ttp)})=card(W^A_{g(\ttp)})=(\max f)(\ttp)$.
\\
Which proves that $\max f$ is in $card\circ{\cal PF}^{RE^A(\bbbx)}$.
\end{proof}
%
%
\subsection{Characterization of the $\Delta\card$
representation system}      \label{ss:Deltacard}

We now look at the self-delimited system with domain $\bbbz$
obtained from $card\circ{\cal F}^{RE^A(\bbbx)}$ by the operation
$\Delta$introduced in \S\ref{ss:Delta}.
\begin{theorem}\label{thm:DeltaCard}
Let $A\subseteq\bbbn$ and let $A'$ be the jump of $A$.
Let $\bbbx$ be a basic set. Then
$$\Delta(card\circ{\cal F}^{RE^A(\bbbx)})
=\Delta(card\circ{\cal PF}^{RE^A(\bbbx)})=\PR[A',\words\to\bbbz]$$
Hence $K^\bbbz_{\Delta(card\circ{\cal F}^{RE^A(\bbbx)})}
\eqct K^{A',\bbbz}$.
\medskip\\
We shall simply write $K^{\bbbn,A}_{\Delta card}$
in place of
$K^\bbbz_{\Delta(card\circ{\cal F}^{RE^A(\bbbn)})}\segment\bbbn$.
\\
When $A=\emptyset$ we simply write $K^{\bbbz}_{\Delta card}$.
\end{theorem}
\begin{proof}
The equalities about the self-enumerated systems is a direct
corollary of Thm.\ref{thm:card} and Thm.\ref{thm:Deltamax}.
The equalities about Kolmogorov complexities are trivial corollaries
of those about self-enumerated systems.
\end{proof}
%
%
\section{Index representations of $\bbbn$}
\label{s:index}
%
%
\subsection{Basic index representation and its effectivizations}
\label{ss:index}
%
A variant of the cardinal representation considers indexes of
equivalence relations. More precisely, it views an integer as
an equivalence class of equivalence relations relative to index
comparison.
\begin{definition}[Index representation]\label{def:index}$\\ $
The index representation of $\bbbn$ relative to an infinite set
$X$ is the partial function
         $$index^{\bbbn}_{P(X^2)}:P(X^2)\to\bbbn$$
with domain the family of equivalence relations on subsets of $X$
which have finite index, such that
\begin{eqnarray*}
index^{\bbbn}_{P(X^2)}(R)&=&
\left\{
\begin{array}{ll}
index(R)&\mbox{if $R$ is an equivalence relation}
     \\ &\mbox{with finite index}
\\
\mbox{undefined}&\mbox{otherwise}
\end{array}\right.
\end{eqnarray*}
(where $index(R)$ denotes the number of equivalence classes of $R$).
\end{definition}
%
%
\subsection{Syntactical complexity of index representations}
\label{ss:syntaxindex}
%
\begin{definition}[Effectivization of the index representation
of $\bbbn$]\label{def:effIndex}
We effectivize the index representation by replacing
$P(X^2)$ by $RE(\bbbx^2)$ or $RE^A(\bbbx^2)$ where $\bbbx$ is
some basic set and $A\subseteq\bbbn$ is some oracle.
\\
Two kinds of self-enumerated representation systems can be
naturally associated
(cf. \S\ref{ss:RE} and the Composition Lemma \ref{l:circ}):%
\begin{eqnarray*}
(RE(\bbbx^2),index\circ{\cal F}^{RE(\bbbx^2)})&\mbox{or}&
(RE^A(\bbbx^2),index\circ{\cal F}^{RE^A(\bbbx^2)})
\\
(RE(\bbbx^2),index\circ{\cal PF}^{RE(\bbbx^2)})&\mbox{or}&
(RE^A(\bbbx^2),index\circ{\cal PF}^{RE^A(\bbbx^2)})
\end{eqnarray*}
\end{definition}
The following proposition gives the syntactical complexity of
the above effectivizations of the index representations.
\begin{proposition}[Syntactical complexity]
\label{p:complexEffIndex}
The family
$$\{domain(\varphi):\varphi\in \indexx\circ{\cal F}^{RE^A(\bbbx)}\}$$
is exactly the family of $\Sigma^{0,A}_3$ subsets of $\words$.
\\
Idem with $\indexx\circ{\cal PF}^{RE^A(\bbbx)}$.
\medskip\\
In particular, any universal function for
$\indexx\circ{\cal F}^{RE^A(\bbbx)}$ or for
$\indexx\circ{\cal PF}^{RE^A(\bbbx)}$
is $\Sigma^{0,A}_3$-complete.
\end{proposition}
\begin{proof}
We trivially reduce to the case $\bbbx=\bbbn$ and only consider
the case $A=\emptyset$, relativization being straightforward.
\medskip\\
1. Let $(W^{\bbbn^2}_\tte)_{\tte\in\words}$ be an acceptable
enumeration of $RE(\bbbn^2)$ and $g:\words\to\words$ be a partial
recursive function and $\psi:\words\to\bbbn$ be such that
$\psi(\ttp)=\indexx(W^{\bbbn^2}_{g(\ttp)})$.
\\
To see that $domain(\psi)$ is $\Sigma^0_3$,
observe that $\ttp\in domain(\psi)$ if and only if
\begin{enumerate}
\item[i.]
$g(\ttp)$ is defined. Which is a $\Sigma^0_1$ condition.
\item[ii.]
$W^{\bbbn^2}_{g(\ttp)}$ is an equivalence relation on its domain,
i.e.
\medskip\\
$\forall \ttx\ \forall \tty\
((\ttx,\tty)\in W^{\bbbn^2}_{g(\ttp)}\
\Rightarrow\ ((\ttx,\ttx)\in W^{\bbbn^2}_{g(\ttp)}\ \wedge\
(\tty,\ttx)\in W^{\bbbn^2}_{g(\ttp)}))$

\hfill{$\wedge\ \forall \ttx\ \forall \tty\ \forall \ttz\
(((\ttx,\tty)\in W^{\bbbn^2}_{g(\ttp)}\ \wedge\
(\tty,\ttz)\in W^{\bbbn^2}_{g(\ttp)})\
\Rightarrow\ (\ttx,\ttz)\in W^{\bbbn^2}_{g(\ttp)})$}
\\
Which is a $\Pi^0_2$ formula
(since $(\ttu,\ttv)\in W^{\bbbn^2}_{g(\ttp)}$ is $\Sigma^0_1$).
\item[iii.]
$W^{\bbbn^2}_{g(\ttp)}$ has finitely many classes, i.e.
$\exists n\ \forall k\ \exists m\leq n\
(k,m)\in W^{\bbbn^2}_{g(\ttp)}$.
Which is a $\Sigma^0_3$ formula.
\end{enumerate}
2. Let $X\subseteq\words$ be $\Sigma^0_3$.
We construct a total recursive function $g:\words\to\words$ such
that $X=\{\ttp: \indexx(W^{\bbbn^2}_{g(\ttp)})\mbox{ is finite}\}$.
\medskip\\
A. Suppose
$X=\{\ttp:\exists u\ \forall v\ \exists w\ R(\ttp,u,v,w)\}$
where $R\subseteq\words\times\bbbn^3$ is recursive.
Let $\theta:\words\times\bbbn^2\to\bbbn$ be the total recursive
function such that
\begin{eqnarray*}
\theta(\ttp,u,t)&=&\mbox{largest $v\leq t$
such that $\forall v'\leq v\ \exists w\leq t\ R(\ttp,u,v',w)\}$}
\end{eqnarray*}
Observe that $\theta$ is monotone increasing with respect to $t$.
Also,
\begin{itemize}
\item[$(*)$]
if $\ttp\notin X$ then, for all $u$,
$\max_{t\in\bbbn}\theta(\ttp,u,t)$ is finite,
\item[$(**)$]
if $\ttp\in X$ and $u$ is least such that
$\forall v\ \exists w\ R(\ttp,u,v,w)$ then
$$\left\{\begin{array}{rcll}
\max_{t\in\bbbn}\theta(\ttp,u,t)&=&+\infty&\\
\max_{t\in\bbbn}\theta(\ttp,u',t)&&\mbox{is finite}&\mbox{for all }u'<u
\end{array}\right.$$
\end{itemize}
Following this observation, given $\ttp\in\words$, we define
a monotone increasing sequence of equivalence relations
$\rho^t_\ttp$ on finite initial intervals of $\bbbn$ such that
$\rho^t_\ttp$ has $t+1$ equivalence classes
$$I^t_{\ttp,0}\ ,\ I^t_{\ttp,1}\ ,\ ...\ ,\ I^t_{\ttp,t}$$
which are successive finite intervals
$$[0,n^t_{\ttp,0}]\ ,\ [n^t_{\ttp,0}+1,n^t_{\ttp,1}]\ ,\
[n^t_{\ttp,1}+1,n^t_{\ttp,2}]\ ,\ \ldots\ ,\
[n^t_{\ttp,t-1}+1,n^t_{\ttp,t}]$$
where
$n^t_{\ttp,1}<n^t_{\ttp,2}<\ldots<n^t_{\ttp,t-1}<n^t_{\ttp,t}$.
\\
The intuition is as follows:
\begin{enumerate}
\item[i.]
the class $I^t_{\ttp,u}$ is related to $\theta(\ttp,u,t)$,
i.e. to the best we can say at step $t$ about the truth value of
$\forall v\ \exists w\ R(\ttp,u,v,w)$.
\item[ii.]
if and when $\theta(\ttp,u,t)$ increases,
i.e. $\theta(\ttp,u,t+1)>\theta(\ttp,u,t)$ for some $u$,
then we increase the class $I^t_{\ttp,u}$ for the least such $u$.
\end{enumerate}
Of course, an equivalence class which grows and remains
an interval either is the rightmost one
or has to aggregate some of its neighbor class(es).
Whence the following inductive definition of the $\rho^t_\ttp$'s
and $n^t_{\ttp,u}$'s, $u\leq t$:
\begin{enumerate}
\item[i.]{\em (Base case).}
$\rho^0_\ttp$ is the equivalence relation with one class $\{0\}$,
i.e. $n^0_\ttp,0=0$.
\item[ii.]{\em (Inductive case. Subcase 1).}
Suppose $\theta(\ttp,u,t+1)=\theta(\ttp,u,t)$ for all $u\leq t$.
Then $\rho^{t+1}_\ttp$ is obtained from $\rho^t_\ttp$ by adding
a new singleton class on the right:
\begin{enumerate}
\item
For all $u\leq t$ we let $n^{t+1}_{\ttp,u}=n^t_{\ttp,u}$,
hence $I^{t+1}_{\ttp,u}=I^t_{\ttp,u}$.
\item
$n^{t+1}_{\ttp,t+1}=n^t_{\ttp,t}+1$, hence
$I^{t+1}_{\ttp,t+1}=\{n^t_{\ttp,t}+1\}$.
\end{enumerate}
\item[ii.]{\em (Inductive case. Subcase 2).}
Suppose $\theta(\ttp,u,t+1)>\theta(\ttp,u,t)$ for some $u\leq t$.
Let $u$ be least such. Then,
\begin{enumerate}
\item
for $u'<u$, classes $I^t_{\ttp,u'}$ are left unchanged:
$n^{t+1}_{\ttp,u'}=n^t_{\ttp,u'}$ and
$I^{t+1}_{\ttp,u'}=I^t_{\ttp,u'}$\ ,
\item
class $I^{t+1}_{\ttp,u}$ aggregates all classes $I^t_{\ttp,u''}$
for $u\leq u''\leq t$,
\item
$t+1-u$ singleton classes are added:
$I^{t+1}_{\ttp,u+i}=\{n^t_{\ttp,t}+i\}$ where $i=1,...,t+1-u$.
I.e.
\medskip\\\centerline{$\begin{array}{rcll}
n^{t+1}_{\ttp,u'}&=&n^t_{\ttp,u}&\mbox{for all }u'\leq u\\
n^{t+1}_{\ttp,u+i}&=&n^t_{\ttp,t}+i
&\mbox{for all $s\in\{i,...,t+1-u\}$}
\end{array}$}
\end{enumerate}
\end{enumerate}
B. Let $\rho_\ttp=\bigcup_{t\in\bbbn}\rho_{\ttp,t}$.
\medskip\\
\fbox{Case $\ttp\in X$}. Let $u$ be least such that
$\forall v\ \exists w\ R(\ttp,u,v,w)$. For $u'<u$, let
\begin{eqnarray*}
V_{u'}&=&\max\{v:\forall v'\leq v\ \exists w\ R(\ttp,u',v',w)\}
\\
t&=&\min\{t':\forall u'<u\ (V_{u'}\leq t'\ \wedge\
\forall v'\leq V_{u'}\ \exists w\leq t'\ R(\ttp,u',v',w)\}
\end{eqnarray*}
Then
\begin{itemize}
\item
$\forall u'<u\ \forall v\
(\forall v'\leq v\ \exists w\ R(\ttp,u',v',w)\ \Rightarrow$

\hfill{$(v\leq t\ \wedge\
\forall v'\leq v\ \exists w'\leq t\ R(\ttp,u',v',w')))$}
\item
$n^{t'}_{\ttp,u',}=n^t_{\ttp,u'}$
and $I^{t'}_{\ttp,u'}=I^t_{\ttp,u'}$ for all $u'<u$ and $t'\geq t$.
\item
$n^{t'}_{\ttp,u}$ tends to $+\infty$ with $t'$ and
$I^{t'}_{\ttp,u}=[n^{t'}_{\ttp,u-1}+1,n^{t'}_{\ttp,u}]$ tends to the
cofinite interval $[n^t_{\ttp,u-1}+1,+\infty[$.
\item
for $u''>u$, classes $I^{t'}_{\ttp,u''}$ are intervals the left
endpoints of which tend to $+\infty$ with $t'$, hence they vanish
at infinity.
\end{itemize}
Thus, $\rho_\ttp$, which is the limit of the $\rho^t_\ttp$'s,
has $u+1$ classes, hence has finite index.
\medskip\\
\fbox{Case $\ttp\notin X$}. For every $u\in\bbbn$,
the class $I^t_{\ttp,u}$ stabilizes as $t$ tends to $+\infty$.
Thus, $\rho_\ttp$ has infinite index.
\medskip\\
C. Clearly, the sequence $(\rho^t_\ttp)_{\ttp\in\words,t\in\bbbn}$
is recursive.
Thus,
$$\rho=\{(\ttp,m,n):\exists t\ (m,n)\in\rho^t_\ttp\}$$
is r.e.
Let $\tta\in\words$ be such that
$\rho=W^{\words\times\bbbn^2}_\tta$.
Applying the parametrization property, let
$s:\words\times\words\to\words$ be a total recursive function
such that
$$\rho_\ttp
=\{(m,n)\in\bbbn^2:(\ttp,m,n)\in W^{\words\times\bbbn^2}_\tta\}
=W^{\bbbn^2}_{s(\tta,\ttp)}$$
Let $g:\words\to\words$ be total recursive such that
$g(\ttp)=s(\tta,\ttp)$.
Using point B, we see that $\ttp\in X$ if and only if
$\indexx(W^{\bbbn^2}_{g(\ttp)})$ is finite.
\end{proof}
%
\subsection{Characterization of the $\indexx$ self-enumerated
systems}  \label{ss:characterizeIndex}

%
We now come to the characterization of the index self-enumerated
families.
It turns out that these families are almost equal to
$\maxrAj[\words\to\bbbn]$, almost meaning here ``up to 1".
\begin{notation}
If ${\cal G}$ is a family of functions
$\words\to\bbbn$, we let
$${\cal G}+1=\{f+1:f\in{\cal G}\}$$
\end{notation}
\begin{theorem}\label{thm:index}$\\ $
{\bf 1.}
For any basic set $\bbbx$ and any oracle $A\subseteq\bbbn$,
the following {\em strict} inclusions hold:
$$\maxrAj[\words\to\bbbn]+1
\subset
index\circ{\cal F}^{RE^A(\bbbx^2)}
\subset
index\circ{\cal PF}^{RE^A(\bbbx^2)}
\subset
\maxrAj[\words\to\bbbn]$$
{\bf 2.}\ \ $K^\bbbn_{index\circ{\cal F}^{RE^A(\bbbx^2)}}\eqct
K^\bbbn_{index\circ{\cal PF}^{RE^A(\bbbx^2)}}\eqct\kmax[A']$.
\medskip\\
We shall simply write $K^{\bbbn,A}_{index}$
in place of $K^\bbbn_{index\circ{\cal F}^{RE^A(\bbbn)}}$.
\\
When $A=\emptyset$ we simply write $K^{\bbbn}_{index}$.
\end{theorem}
\begin{proof}
Observe that if ${\cal F}$ is a self-enumerated system with domain
$D$ and with $U$ as a good universal function, then ${\cal F}+1$
is also a self-enumerated system with $U+1$ as a good universal
function.
In particular $K^D_{\cal F}=K^D_{{\cal F}+1}$.
\\
Point 2 is a direct corollary of Point 1 and Prop.\ref{p:Kinfini}
and the previous observation.
\medskip\\
Let's prove point 1.
\\The central inclusion
$index\circ{\cal F}^{RE^A(\bbbx^2)}
\subset index\circ{\cal PF}^{RE^A(\bbbx^2)}$
is trivial.
\medskip\\
A. {\em Non strict inclusion\ \
$index\circ{\cal PF}^{RE^A(\bbbx^2)}
\subseteq\maxrAj[\words\to\bbbn]$.}
\\
Let $G\in index\circ{\cal PF}^{RE^A(\bbbx^2)}$ and let
$g:\words\to\words$ be partial $A$-recursive such that
$$G(\ttp)=\left\{\begin{array}{ll}
index(W^{A,\bbbx^2}_{g(\ttp)})&
\mbox{if $g(\ttp)$ is defined and $W^{A,\bbbx^2}_{g(\ttp)}$ is an}
\\&\mbox{equivalence relation with finite index}
\\\mbox{undefined}&\mbox{otherwise}\end{array}\right.$$
We define a total $A'$-recursive function
$u:\words\times\bbbn\to\bbbn$ such that
$$(*)\ \ \ \{u(\ttp,t):t\in\bbbn\}=
\left\{\begin{array}{ll}
\{0,...,n\}&\mbox{if $G(\ttp)$ is defined and $G(\ttp)=n$}\\
\bbbn&\mbox{if $G(\ttp)$ is undefined}
\end{array}\right.$$
\noindent
The definition is as follows.
Since $g$ is partial $A$-recursive and we look for an $A'$-recursive
definition of $u(\ttp,t)$, we can use oracle $A'$ to check
if $g(\ttp)$ is defined.
\\
If $g(\ttp)$ is undefined then we let $u(\ttp,t)=t$ for all $t$.
Which insures $(*)$.
\\
Suppose now that $g(\ttp)$ is defined.
First, set $u(\ttp,0)=0$.
\\
Consider an $A$-recursive enumeration of $W^{A,\bbbx^2}_{g(\ttp)}$.
Let $R_t$ be the set of pairs enumerated at steps $<t$
and $D_t$ be the set of $\ttx\in\bbbx$ which appear in pairs
in $R_t$ (so that $R_0$ and $D_0$ are empty).
Since at most one new pair is enumerated at each step, the set
$R_t$ contains at most $t$ pairs and $D_t$ contains at most $2\,t$
points.
\\
At step $t+1$, use oracle $A'$ to check the following properties:
\begin{enumerate}
\item[$\alpha_t$.]
For every $\ttx\in D_{t+1}$ the pair $(\ttx,\ttx)$ is in
$W^{A,\bbbx^2}_{g(\ttp)}$.
\item[$\beta_t$.]
For every pair $(\ttx,\tty)\in R_{t+1}$ the pair $(\tty,\ttx)$
is in $W^{A,\bbbx^2}_{g(\ttp)}$.
\item[$\gamma_t$.]
For every pairs $(\ttx,\tty),(\tty,\ttz)\in R_{t+1}$ the pair
$(\ttx,\ttz)$ is in $W^{A,\bbbx^2}_{g(\ttp)}$.
\item[$\delta_t$.]
For every $\ttx\in D_{t+1}$ there exists $\tty\in D_t$ such that
the pair $(\ttx,\tty)$ is in $W^{A,\bbbx^2}_{g(\ttp)}$.
\end{enumerate}
Since $R_{t+1},D_{t+1}$ are finite, all these properties
$\alpha_t\tiret\delta_t$ are {\em finite} boolean combinations of
$\Sigma^{0,A}_1$ statements.
Hence oracle $A'$ can decide them all.
\medskip\\
Observe that if $W^{A,\bbbx^2}_{g(\ttp)}$ is an equivalence relation
then answers to $\alpha_t\tiret\gamma_t$ are positive for all $t$.
And if $W^{A,\bbbx^2}_{g(\ttp)}$ is not an equivalence relation
then, for some $\pi\in\{\alpha,\beta,\gamma\}$, answers to $\pi_t$
are negative for all $t$ large enough .
\\
Also, if $W^{A,\bbbx^2}_{g(\ttp)}$ is an equivalence relation then
a new equivalence class is revealed each time $\delta_t$ is false.
And every equivalence class is so revealed.
\medskip\\
Thus, in case $g(\ttp)$ is defined, we insure $(*)$ by letting
$$u(\ttp,t+1)=\left\{
\begin{array}{ll}
u(\ttp,t)
&\mbox{if all answers to $\alpha_t\tiret\delta_t$ are positive}
\\
u(\ttp,t)+1&\mbox{otherwise}
\end{array}\right.$$
From $(*)$, we get $G=\max u$. Since $u$ is total $A'$-recursive,
this proves that $G$ is in $\maxrAj[\words\to\bbbx]$.
\medskip\\
B. {\em Non strict inclusion\ \
$\maxrAj[\words\to\bbbn]+1
\subseteq index\circ{\cal F}^{RE^A(\bbbx^2)}$.}\\
We reduce to the case $\bbbx=\bbbn$.
\\Let $F\in\maxprAj[\words\to\bbbn]$. Using Prop.\ref{p:Turing},
let ${\cal M}$ be an oracle Turing machine which on input $\ttp$
and oracle $A'$ computes $F(\ttp)$ through an infinite computation.
\\
The idea to prove that $F$ is in
$index\circ{\cal F}^{RE^A(\bbbn^2)}$ is as follows.
We consider $A$-recursive approximations of oracle $A'$ and use
them as fake oracles.
For each $\ttp$ we build an $A$-r.e. equivalence relation
$\rho_\ttp\subseteq\bbbn^2$ with domain $\bbbn$ which consists of
one big class containing $0$ and some singleton classes.
Each time the computation with the fake oracle outputs a new digit
$1$, we put some new singleton class in $\rho_\ttp$.
When, with a better approximation of $A'$, we see that the fake
oracle has given an incorrect answer, all singleton classes
which were put in $\rho_\ttp$ because of the oracle incorrect answer
are annihilated: they are aggregated to the class of $0$.
Since we are going to consider $index(\rho_\ttp)$, this process
will lead to the correct value $F(\ttp)+1$.
\medskip\\
Formally, we consider an $A$-recursive monotone increasing sequence
$(Approx(A',t))_{t\in\bbbn}$ such that
$A'=\bigcup_{t\in\bbbn}Approx(A',t)$ (cf. Lemma \ref{l:approx}).
Though all oracles $Approx(A',t)$ are false approximations of oracle
$A'$, they are nevertheless ``less and less false" as $t$ increases.
\medskip\\
Without loss of generality, we can suppose that at each computation
step of ${\cal M}$ there is a question to the oracle
(possibly the same one many times).
\medskip\\
Let ${\cal C}_{\ttp,t}$ be the computation of ${\cal M}$ on input
$\ttp$ with oracle $Approx(A',t)$, reduced to the sole $t$ first
steps.
\\
Increasing parts of oracle $Approx(A',t)$ are questioned during
${\cal C}_{\ttp,t}$.
Let $\Omega_{\ttp,t}:\{1,...,t\}\to P_{fin}(\bbbn)$
(where $P_{fin}(\bbbn)$ is the set of finite subsets of $\bbbn$)
be such that $\Omega_{\ttp,t}(t')$ is the set of $k$ such that the
oracle has been questioned about $k$ during the $t'$ first steps,
$1\leq t'\leq t$.
Clearly, $\Omega_{\ttp,t}$ is (non strictly) monotone increasing
with respect to set inclusion.
\\
Let $1^{n_{\ttp,t}}$ be the output of ${\cal C}_{\ttp,t}$
(recall that ${\cal M}$ outputs a finite or infinite sequence of
digits $1$'s).
\\
The successive digits of this output are written down at increasing
times (all $\leq t$).
Let $OT_{\ttp,t}:\{0,...,n_{\ttp,t}\}\to\{0,...,t\}$ be such that
$OT_{\ttp,t}(n)$ is the least step at which the current output is
$1^n$ ($OT$ stands for output time). Clearly, $OT_{\ttp,t}(0)=0$.
\medskip\\
We construct $A$-recursive sequences
$(\rho_{\ttp,t})_{\ttp\in\words, t\in\bbbn}$ and
$(w_{\ttp,t})_{\ttp\in\words, t\in\bbbn}$
(where $w$ stands for witness) such that
\begin{enumerate}
\item[$i_t$.]
$\rho_{\ttp,t}$ is an equivalence relation on $\{0,...,2^t-1\}$
with index equal to $1+n_{\ttp,t}$
(there is nothing essential with $2^t$, it is merely a large enough
bound convenient for the construction),
\item[$ii_t$.]
all equivalence classes of $\rho_{\ttp,t}$ are singleton sets
except possibly the equivalence class of $0$.
\item[$iii_t$.]
if $t>0$ then $\rho_{\ttp,t}$ contains $\rho_{\ttp,t-1}$.
\item[$iv_t$.]
$w_{\ttp,t}$ is a bijection between $\{1,...,n_{\ttp,t}\}$
and the set of point $s\in\{1,...,2^t-1\}$ such that $\{s\}$ is a
singleton class of $\rho_{\ttp,t}$
(in case $n_{\ttp,t}=0$ then $w_{\ttp,t}$ is the empty map).
\end{enumerate}
First, $w_{\ttp,0}$ is the empty map and $\rho_{\ttp,0}=\{(0,0)\}$,
i.e. the trivial equivalence relation on $\{0\}$.
\medskip\\
The inductive construction of the $\rho_{\ttp,t}$'s uses the
above conditions $i_t\tiret iv_t$ as an induction hypothesis.
\medskip\\
{\em Case $Approx(A',t+1)\cap\Omega_{\ttp,t}(t)
=Approx(A',t)\cap\Omega_{\ttp,t}(t)$.}
\\
Then the computation ${\cal C}_{\ttp,t}$ is totally compatible with
${\cal C}_{\ttp,t+1}$.
Now, that last computation may possibly output one more digit $1$,
i.e. $n_{\ttp,t+1}=n_{\ttp,t}$ or $n_{\ttp,t+1}=n_{\ttp,t}+1$.
Hence the two following subcases.
\medskip\\
{\em Subcase $n_{\ttp,t+1}=n_{\ttp,t}$.}
Then $\rho_{\ttp,t+1}$ is obtained from $\rho_{\ttp,t}$ by putting
$2^t,2^t+1,...,2^{t+1}-1$ as new points in the class of $0$.
In particular,
$\rho_{\ttp,t+1}$ and $\rho_{\ttp,t}$ have the same index.
We also set $w_{\ttp,t+1}=w_{\ttp,t}$.
\medskip\\
{\em Subcase $n_{\ttp,t+1}=n_{\ttp,t}+1$.}
Then $\rho_{\ttp,t+1}$ is obtained from $\rho_{\ttp,t}$ as follows:
\begin{itemize}
\item
Add a new singleton class $\{2^t\}$.
\item
Put $2^t+1,...,2^{t+1}-1$ as new points in the class of $0$.
\end{itemize}
We also set $w_{\ttp,t+1}=w_{\ttp,t}\cup\{(n_{\ttp,t+1},2^t)\}$.
\medskip\\
In both subcases, conditions $i_{t+1}\tiret iv_{t+1}$ are
clearly satisfied.
\medskip\\
{\em Case $Approx(A',t+1)\cap\Omega_{\ttp,t}(t)
\neq Approx(A',t)\cap\Omega_{\ttp,t}(t)$.}
\\
Let $\tau\leq t$ be least such that
$Approx(A',t+1)\cap\Omega_{\ttp,t}(\tau)
\neq Approx(A',t)\cap\Omega_{\ttp,t}(\tau)$.
Though the computation ${\cal C}_{\ttp,t}$ is not entirely
compatible with ${\cal C}_{\ttp,t+1}$, it is compatible up to
step $\tau-1$.
\\
Let $n\leq n_{\ttp,t}$ be greatest such that $OT_{\ttp,t}(n)<\tau$.
Then the $n$ first digits output by ${\cal C}_{\ttp,t}$ are also
output by ${\cal C}_{\ttp,t+1}$ at the same computation steps.
In particular, $n_{\ttp,t+1}\geq n$.
\\
Then $\rho_{\ttp,t+1},w_{\ttp,t+1}$ are obtained from
$\rho_{\ttp,t},w_{\ttp,t}$ as follows:
\begin{itemize}
\item
Put all $w_{\ttp,t}(m)$, where $n<m\leq n_{\ttp,t}$, as new points
in the class of $0$. This annihilates the singleton classes of
$\rho_{\ttp,t}$ corresponding (via $w_{\ttp,t}(m)$) to the part of
the output which was created by answers of oracle $Approx(A',t)$
which are known to be false at step $t+1$.
\item
Add a new singleton class $\{2^t-1+i\}$ for each $i>0$ such that
$n+i\leq n_{\ttp,t+1}$.
Together with the singleton classes of $\rho_{\ttp,t}$ which have
not been aggragated by the above point, this allows to get exactly
$n_{\ttp,t+1}$ singleton classes in $\rho_{\ttp,t+1}$
\\
Accordingly, set
$$w_{\ttp,t+1}=(w_{\ttp,t}\segment \{1,...,n\})\
\cup\ \{(n+i,2^t-1+i):0<i\leq n_{\ttp,t+1}-n\}$$
\item
Put the $2^t-1+j$'s, where $j\geq\max(1,n_{\ttp,t+1}-n)$,
as new points in the class of $0$.
\end{itemize}
Again, conditions $i_{t+1}\tiret iv_{t+1}$ are clearly satisfied.
\medskip\\
Let $\rho_\ttp=\bigcup_{t\in\bbbn}\rho_{\ttp,t}$.
Condition $iii_t$ insures that $\rho_\ttp$ is also an equivalence
relation.
Condition $ii_t$ goes through the limit when $t\to+\infty$,
so that all classes of $\rho_\ttp$ are singleton sets except
the class of $0$.
\medskip\\
The computation we are really interesting in is that which gives
$F(\ttp)$, i.e. the infinite computation of ${\cal M}$
on input $\ttp$ with oracle $A'$.
Let denote it ${\cal C}_\ttp$.
When $t$ increases, the common part of ${\cal C}_\ttp$ with
computation ${\cal C}_{\ttp,t}$ gets larger and larger
(though not monotonously).
\medskip\\
We now prove the equality
$$(\dagger)\ \ \ index(\rho_\ttp)=\left\{\begin{array}{ll}
1+F(\ttp)&\mbox{if $F(\ttp)$ is defined}\\
+\infty&\mbox{otherwise}
\end{array}\right.$$
\medskip\\
{\em Case $F(\ttp)$ is defined and $F(\ttp)=z$.}
\\Let $\tau$ be the computation time at which ${\cal C}_\ttp$ has
output $z$. Let $\Omega_\ttp$ be the set of $k$ such that oracle
$A'$ has been questioned about during the first $\tau$ steps of
${\cal C}_\ttp$.
For $t$ large enough, say $t\geq t_z$, we have
$Approx(A',t)\cap\Omega_\ttp=A'\cap\Omega_\ttp$.
In particular, the $\tau$ first steps of ${\cal C}_{\ttp,t}$
and ${\cal C}_\ttp$ will be exactly the same and both computations
output $z$.
The same with the $\tau$ first steps of ${\cal C}_{\ttp,t}$
and ${\cal C}_{\ttp,t+1}$.
\\Thus,
$w_{\ttp,t+1}\segment\{1,...,z\}=w_{\ttp,t}\segment\{1,...,z\}$.
\\
Let $w_\ttp=w_{\ttp,t+1}\segment\{1,...,z\}$.
Then all singleton sets $\{w_\ttp(i)\}$, where $1\leq i\leq z$,
are equivalence classes for the $\rho_{\ttp,t}$'s, hence for
$\rho_\ttp$.
\medskip\\
Now, if $n_{\ttp,t}>z$ then oracle $Approx(A',t)$ has been
questioned on $\Omega_{\ttp,t}(n_{\ttp,t})$ and differs from $A'$
on that set.
Let $u>t$ be first such that $Approx(A',u)$ agrees with $A'$
on $\Omega_{\ttp,t}(z+1)$.
Then the singleton class $\{w_{\ttp,t}(z+1)\}$ of
$\rho_{\ttp,t}$ is aggregated at step $u$ to the class of $0$
in $\rho_{\ttp,t+1}$, hence also in $\rho_\ttp$.
\medskip\\
Thus, the $\{w_\ttp(i)\}$'s, where $1\leq i\leq z$,
are the sole singleton equivalence classes of $\rho_\ttp$.
And the class of $0$ contains all other points in $\bbbn$.
\\
In particular, $index(\rho_\ttp)=1+F(\ttp)$.
\medskip\\
{\em Case $F(\ttp)$ is undefined because the output of ${\cal M}$
on input $\ttp$ with oracle $A'$ is infinite.}
\\
As in the above case, we see that there are more and more singleton
set classes of $\rho_{\ttp,t}$ which are never annihilated.
Thus, the index of $\rho_\ttp$ is infinite.
\medskip\\
This proves $(\dagger)$.
\medskip\\
Observing that all the construction of the $\rho_{\ttp,t}$'s is
$A$-recursive, we see that
$$\rho=\bigcup_{\ttp\in\words}\rho_\ttp$$
is $A$-r.e.
Thus, $\rho=W^{A,\words\times\bbbn^2}_\tta$ for some $\tta$.
The parameter property gives a total $A$-recursive function
$s:\words\times\words\to\words$ such that
$$\rho_\ttp=W^{A,\bbbn^2}_{s(\tta,\ttp)}$$
Thus, $p\mapsto index(\rho_\ttp)$ is indeed in
$index\circ{\cal F}^{RE^A(\bbbx^2)}$.
Thanks to $(\dagger)$, the same is true of $1+F$.
\medskip\\
C. {\em Inclusion\ \ $\maxrAj[\words\to\bbbn]+1
\subseteq index\circ{\cal F}^{RE^A(\bbbx^2)}$ is strict.}\\
The constant $0$ function is an obvious counterexample
to equality.
\medskip\\
D. {\em Inclusion\ \ $index\circ{\cal PF}^{RE^A(\bbbx^2)}
\subseteq\maxrAj[\words\to\bbbn]$ is strict.}\\
We exhibit a function $\kappa_X$ in
$\PR[A']\setminus index\circ{\cal PF}^{RE^A(\bbbx^2)}\neq\emptyset.$
\\
Let $X\subset\words$ be $A'$-recursive, i.e. $\Delta^{0,A}_2$,
but not a boolean combination of $\Sigma^{0,A}_1$ sets.
Let $\kappa_X:\words\to\bbbn$ be the $\{0,1\}$-valued
characteristic function of $X$.
Then $\kappa_X$ is $A'$-recursive
(hence in $\maxrAj[\words\to\bbbn]$)
and $\kappa_X^{-1}(0)=X$ is a $\Delta^{0,A}_2$ set which is not
a boolean combination of $\Sigma^{0,A}_1$ sets.
\medskip\\
Now, suppose $G$ is in $index\circ{\cal PF}^{RE^A(\bbbx^2)}$ and
$G=index(W^{A,\bbbx^2}_{g(\ttp)})$ where $g:\words\to\words$ is
in $\PR[A]$. Then
\begin{eqnarray*}
G(\ttp)=0&\Leftrightarrow&(\mbox{$g(\ttp)$ is defined}\
\wedge\ W^{A,\bbbx^2}_{g(\ttp)}=\emptyset)
\\
&\Leftrightarrow&(\mbox{$g(\ttp)$ is defined}\\
&&\wedge\ \forall t\ \forall \tte\
(g(\ttp) \mbox{ converges to $\tte$ in $t$ steps }
\Rightarrow\ W^{A,\bbbx^2}_\tte=\emptyset)
\end{eqnarray*}
so that $G^{-1}(0)$ is $\Sigma^{0,A}_1\wedge\Pi^{0,A}_1$.
\medskip\\
This shows that no $G\in index\circ{\cal PF}^{RE^A(\bbbx^2)}$ can be
equal to the above $\kappa_X$.
Therefore, the considered inclusion cannot be an equality.
\end{proof}
Let's finally observe a simple fact contrasting inclusions in
Thm.\ref{thm:index}.
\begin{proposition}\label{p:indexnoninclusion}
$1+\PR[A',\words\to\words]$ (a fortiori $1+\maxprAj[\words\to\bbbn]$)
is not included in $index\circ{\cal PF}^{RE^A(\bbbx^2)}$.
\end{proposition}
\begin{proof}
The proof is analog to that of point D in the proof of
Thm.\ref{thm:index}.
\medskip\\
1. We show that $G^{-1}(1)$ is $\Pi^{0,A}_2$ for every
$G\in index\circ{\cal PF}^{RE^A(\bbbx^2)}$.
\\
Suppose $G=index(W^{A,\bbbx^2}_{g(\ttp)})$ where $g:\words\to\words$
is partial $A$-recursive.
\\
Let's denote $W^{A,\bbbx^2}_{\tte,t}$ the finite part of
$W^{A,\bbbx^2}_\tte$ obtained after $t$ steps of its enumeration.
Let's also denote $CV_g(\ttp,\tte,t)$ the $A$-recursive relation
stating that $g(\ttp)$ converges to $\tte$ in $\leq t$ steps.
Then
\begin{eqnarray*}
G(\ttp)=1&\Leftrightarrow&(\mbox{$g(\ttp)$ is defined}\
\wedge\ W^{A,\bbbx^2}_{g(\ttp)}\neq\emptyset\
\\&&\wedge\ W^{A,\bbbx^2}_{g(\ttp)}
\mbox{ is an equivalence relation with index $1$})
\\
&\Leftrightarrow&(\mbox{$g(\ttp)$ is defined}\
\wedge\ W^{A,\bbbx^2}_{g(\ttp)}\neq\emptyset\
\\
&&\wedge\ \forall t\ \forall \tte\
(CV_g(\ttp,\tte,t)\ \Rightarrow
\\&&\hspace{2cm}W^{A,\bbbx^2}_\tte
\mbox{ is an equivalence relation with index $1$})
\end{eqnarray*}
The first two conjuncts are clearly $\Sigma^{0,A}_1$.
As for the last one, observe that $W^{A,\bbbx^2}_\tte$
is an equivalence relation if and only if
\begin{eqnarray*}
&\forall \ttx,\tty\in\bbbx\
((\ttx,\tty)\in W^{A,\bbbx^2}_\tte\ \Rightarrow\
(\ttx,\ttx)\in W^{A,\bbbx^2}_\tte\ \wedge\
(\tty,\ttx)\in W^{A,\bbbx^2}_\tte)&
\\
&\wedge\ \forall \ttx,\tty,\ttz\in\bbbx\
((\ttx,\tty)\in W^{A,\bbbx^2}_\tte\ \wedge\
(\tty,\ttz)\in W^{A,\bbbx^2}_\tte)\ \Rightarrow\
(\ttx,\ttz)\in W^{A,\bbbx^2}_\tte)&
\end{eqnarray*}
Which is $\Pi^{0,A}_2$ since $W^{A,\bbbx^2}_\tte$ is
$\Sigma^{0,A}_1$.
\\
Also, if $W^{A,\bbbx^2}_\tte$ is a non empty equivalence relation
then it has index $1$ if and only if
\begin{eqnarray*}
&\forall \ttx,\tty,\ttx',\tty'\in\bbbx\
((\ttx,\ttx')\in W^{A,\bbbx^2}_\tte\ \wedge\
(\tty,\tty')\in W^{A,\bbbx^2}_\tte,)\ \Rightarrow\
(\ttx,\tty)\in W^{A,\bbbx^2}_\tte)&
\end{eqnarray*}
Which is again $\Pi^{0,A}_2$.
\medskip\\
This proves that $G^{-1}(1)$ is indeed $\Pi^{0,A}_2$.
\medskip\\
2. Now, let $X\subset\bbbx$ be $\Sigma^{0,A'}_1$ and not
$A'$-recursive.
Thus, $X$ is $\Sigma^{0,A}_2$ and not $\Pi^{0,A}_2$.
Let $\pi_X:\words\to\bbbn$ be such that
$$\pi_X(\ttp)=\left\{\begin{array}{ll}
1&\mbox{if $\ttp\in X$}\\
\mbox{undefined}&\mbox{otherwise}\end{array}\right.$$
Then $\pi_X\in1+\PR[A',\words\to\bbbn]$.
\medskip\\
Since $\pi_X^{-1}(1)=X$ is not $\Pi^{0,A}_2$, $\pi_X$ cannot be in
$index\circ{\cal PF}^{RE^A(\bbbx^2)}$.
\end{proof}
%
%
\subsection{Characterization of the $\Delta\indexx$ self-enumerated
systems}  \label{ss:DeltaIndex}

\begin{theorem}\label{thm:DeltaIndex}$\\ $
Let $A\subseteq\bbbn$ and let $A''$ be the second jump of $A$.
Let $\bbbx$ be a basic set.
\medskip\\
{\bf 1.}\ \
$\Delta(index\circ{\cal F}^{RE^A(\bbbx)}))
=\Delta(index\circ{\cal PF}^{RE^A(\bbbx)}))
=\PR[A'',\words\to\bbbz]$
\medskip\\
{\bf 2.}
$K^\bbbz_{\Delta(index\circ{\cal F}^{RE^A(\bbbx)})}
\eqct K^{A'',\bbbz}$.
\medskip\\
We shall simply write $K^{\bbbn,A}_{\Delta index}$
in place of
$K^\bbbz_{\Delta(index\circ{\cal F}^{RE^A(\bbbn)})}\segment\bbbn$.
\\
When $A=\emptyset$ we simply write $K^{\bbbz}_{\Delta index}$.
\end{theorem}
\begin{proof}
Point 2 is a direct corollary of Point 1.
Let's prove point 1.
Using Thm.\ref{thm:index}, and applying the $\Delta$ operator,
we get
\medskip\\
$\Delta(\maxrAj[\words\to\bbbn]+1)
\subseteq
\Delta(index\circ{\cal F}^{RE^A(\bbbx^2)})$

\hfill{$\subseteq
\Delta(index\circ{\cal PF}^{RE^A(\bbbx^2)})
\subseteq
\Delta(\maxrAj[\words\to\bbbn])$}
\medskip\\
But, for any family ${\cal G}$ of functions $\words\to\bbbn$,
we trivially have $\Delta({\cal G}+1)=\Delta({\cal G})$.
This proves that the above inclusions are, in fact, equalities.
We conclude with Thm.\ref{thm:Deltamax}.
\end{proof}
%
%
%
\section{Functional representations of $\bbbn$} \label{s:church}
%
\begin{notation}[Functions sets]\label{not:general2}
We denote
\\\indent -\ \ $Y^X$ the set of total functions from $X$ into $Y$.
\\\indent -\ \ $X\to Y$ the set of partial functions from $X$
               into $Y$.
\\\indent -\ \ $X\stackrel{1-1}{\to}X$ the set of injective
               partial functions from $X$ into $X$.
\\\indent -\ \ $Id_X$ the identity function over $X$.
\end{notation}
%
%
\subsection{Basic Church representation of $\bbbn$}
\label{ss:church}
%
First, let's introduce some simple notations related to function
iteration.
\begin{definition}[Iteration]\label{def:it}$\\ $
1) If $f:X\to X$ is a partial function, we inductively define
for $n\in\bbbn$ the $n$-th iterate $f^{(n)}:X\to X$ of $f$
as the partial function such that:
     $$f^{(0)}=Id_X \ ,\ f^{(n+1)}=f^{(n)}\circ f$$
\medskip
2) $It^{(n)}_X:(X\to X)\to(X\to X)$
is the total functional $f\mapsto f^{(n)}$.
\\
$It_X^\bbbn:\bbbn \to (X\to X)^{(X\to X)}$
is the total functional $n\mapsto It_X^n$.
\end{definition}
The following Proposition is easy.
\begin{proposition}\label{p:injectiveIt}
The total functional $It^\bbbn_X:\bbbn \to (X\to X)^{(X\to X)}$
is injective (hence admits a left inverse) if and only if $X$ is
an infinite set.
\end{proposition}
We can now come to the functional representation of integers
introduced by Church, 1933 \cite{church33}.
\begin{definition}[Church representation of $\bbbn$]
\label{def:Church}$\\ $
If $X$ is an infinite set, the Church representation of $\bbbn$
relative to $X$ is the function
$$\Church^{\bbbn}_X:(X\to X)^{(X\to X)}\to\bbbn$$
which is the unique left inverse of $It_X^\bbbn$ with
domain $Range(It_X^\bbbn)=\{It_X^n:n\in\bbbn\}$, i.e.
\begin{eqnarray*}
\Church^{\bbbn}_X\circ It_X^\bbbn&=&Id_\bbbn\\
\Church^{\bbbn}_X(F)&=&\left\{
\begin{array}{ll}
               n&\mbox{if $F=It_X^n$}\\
\mbox{undefined}&\mbox{if $\forall n\in\bbbn\ F\neq It_X^n$}
\end{array}\right.
\end{eqnarray*}
\end{definition}
For future use in Def.\ref{def:effChurch}, let's introduce the following
variant of $\Church^{\bbbn}_X$.
\begin{definition}\label{def:church}
We denote
$church^{\bbbn,A}_X:(\PR[A,\bbbx\to\bbbx])^{\PR[A,\bbbx\to\bbbn]}$
the functional which is the unique left inverse of the restriction
of $It^\bbbn_X$ to $(\PR[A,\bbbx\to\bbbx])^{\PR[A,\bbbx\to\bbbx]}$,
i.e.
\begin{eqnarray*}
\church^{\bbbn,A}_X(F)&=&\left\{
\begin{array}{ll}
n&\mbox{if }
F=It_X^n\segment(\PR[A,\bbbx\to\bbbx])^{\PR[A,\bbbx\to\bbbx]}\\
\mbox{undefined}&\mbox{if }
\forall n\in\bbbn\ F\neq
It_X^n\segment(\PR[A,\bbbx\to\bbbx])^{\PR[A,\bbbx\to\bbbx]}
\end{array}\right.
\end{eqnarray*}
\end{definition}
%
\subsection{Computable and effectively continuous functionals}
\label{ss:computfunctionals}
%
We recall the two classical notions of partial computability
for functionals,
cf. Odifreddi's book \cite{odifreddi} p.178, 188, 197.
\begin{definition}[Kleene partial computable functionals]
\label{def:Kleene}$\\ $
{\bf 1.}
Let $\bbbx,\bbby,\bbbs,\bbbt$ be some basic space and fix some
suitable representations of their elements by words.
An $(\bbbx\to\bbby)$-oracle Turing machine with inputs and outputs
respectively in $\bbbs,\bbbt$ is a Turing machine
${\cal M}$ which has a special oracle tape and is allowed at
certain states to ask an oracle $f\in(\bbbx\to\bbbx)$ what are
the successive digits of the value of $f(\ttq)$ where $\ttq$
is the element of $\bbbx$ currently written on the oracle tape.
\\
The functional $\Phi_{\cal M}:((\bbbx\to\bbby)\times\bbbs)\to\bbbt$
associated to ${\cal M}$ maps the pair $(f,\tts)$ on the output
(when defined) computed by ${\cal M}$ when $f$ is given as the partial
function oracle and $\tts$ as the input.
\\
If on input $\ttx$ and oracle $f$ the computation asks the oracle
its value on an element on which $f$ is undefined then ${\cal M}$
gets stuck, so that $\Phi_{\cal M}(f,\ttx)$ is undefined.
\medskip\\
{\bf 2.}
A functional $\Phi:((\bbbx\to\bbby)\times\bbbs)\to\bbbt$ is partial
computable (also called partial recursive) if $\Phi=\Phi_{\cal M}$
for some ${\cal M}$.
\\
A functional obtained via curryfications from such a functional is
also called partial computable.
\medskip\\
We denote $\KleenePC^\tau$ the family of partial computable
functionals with type $\tau$.
\\
If $A\subseteq\bbbn$, we denote $A\tiret\KleenePC^\tau$ the analog
family with the extra oracle $A$.
\end{definition}
\begin{definition}[Uspenskii (effectively) continuous functionals]
\label{def:Uspenskii}
Denote $Fin(\bbbx\to\bbby)$ the class of partial functions
$\bbbx\to\bbby$ with finite domains.
Observe that, for $\alpha,\beta\in Fin(\bbbx\to\bbby)$ are
compatible if and only if $\alpha\cup\beta\in Fin(\bbbx\to\bbby)$.
\medskip\\
{\bf 1.}
Let's say that the relation
$R\subseteq Fin(\bbbx\to\bbby)\times\bbbs\times\bbbt$
is functional if
$$\alpha\cup\beta\in Fin(\bbbx\to\bbby)\ \wedge\
(\alpha,\tts,\ttt)\in R\ \wedge\ (\beta,\tts,\ttt')\in R
\ \Rightarrow\ \ttt=\ttt'$$
To such a functional relation $R$ can be associated a functional
$$\Phi_R:((\bbbx\to\bbby)\times\bbbs)\to\bbbt$$
such that, for every $f,\tts,\ttt$,
\medskip\\\medskip\centerline{
$\begin{array}{crcl}
(\dagger)\indent\indent&\Phi(f,\tts)=\ttt&\Leftrightarrow&
\exists u\subseteq f\ R(u,\tts,\ttt)
\end{array}$}
{\bf 2.} {\bf (Uspenskii \cite{uspenskii55}, Nerode \cite{nerode57})}
A functional $\Phi:((\bbbx\to\bbby)\times\bbbs)\to\bbbt$ is
{\bf continuous} if it is of the form $\Phi_R$ for some functional
relation $R$.
\medskip\\
$\Phi$ is {\bf effectively continuous
(resp. ($A$-effectively continuous)}
if $R$ is r.e. (resp. $A$-r.e.).
Effectively continuous functionals are also called recursive
operators (cf. Rogers \cite{rogers}, Odifreddi \cite{odifreddi}).
\\
A functional obtained via curryfications from such a functional is
also called effectively continuous.
\\
We denote $\EffCont^\tau$ the family of effectively continuous
functionals with type $\tau$.
\\
If $A\subseteq\bbbn$, we denote $A\tiret\EffCont^\tau$ the analog
family with the extra oracle $A$.
\end{definition}
Effective continuity is more general than partial computability
(cf. \cite{odifreddi} p.188).
\begin{theorem}\label{thm:KleeneUspenskii}
Let $A\subseteq\bbbn$.\\
{\bf 1.} {\bf (Uspenskii \cite{uspenskii55}, Nerode \cite{nerode57})}
Partial $A$-computable functionals are $A$-effectively continuous.
\medskip\\
{\bf 2.} {\bf (Sasso \cite{sasso71,sasso75})}
There are $A$-effectively continuous functionals which are not
partial $A$-computable.
\end{theorem}
However, restricted to total functions, both notions coincide.
\begin{proposition}
A functional $\Phi:(\bbby^{\bbbx})\times\bbbs\to\bbbt$
is the restriction of a partial $A$-computable functional
$((\bbbx\to{\bbby})\times\bbbs)\to\bbbt$ if and only if it is
the restriction of an $A$-effectively continuous functional.
\end{proposition}
%
%
\subsection{Effectiveness of the {\em Apply} functional}
\label{ss:apply}

The following result will be used in
\S\ref{ss:syntaxChurch}-\ref{ss:effectiveChurch}.
\begin{proposition}\label{p:apply}
Let $\phi:\words\to \PR[A,\bbbx\to\bbbx]$ be partial $A$-recursive
(as a function $\words\times\bbbx\to\bbbx$)
and $\Phi:\words\to
A\tiret\EffCont^{\words\to((\bbbx\to\bbbx)\to(\bbbx\to\bbbx))}$
be effectively continuous.
There exists a partial $A$-recursive function
$g:\words\times\words\times\bbbx$ such that,
for all $\tte,\ttp\in\words$ and $\ttx\in\bbbx$,
$$(*)\indent\indent\indent
g(\ttp,\tte,\ttx)=(\Phi(\tte)(\phi(\ttp)))(\ttx)$$
\end{proposition}
\begin{proof}
Let $R\subseteq
\words\times Fin(\bbbx\to\bbbx)\times\bbbx\times\bbbx$
be an $A$-r.e. set such that, for all $\tte$,
$R^{(\tte)}=\{(\alpha,\ttx,\tty):(\tte,\alpha,\ttx,\tty)\in R\}$
is functional and $\Phi(\tte)=\Phi_{R^{(\tte)}}$.
We define $g(\ttp,\tte,\ttx)$ as follows:
\begin{enumerate}
\item[i.]
$A$-effectively enumerate $R^{(\tte)}$ and the graph of
$\phi(\ttp)$ up to the moment we get
$(\alpha,\ttx,\tty)\in R^{(\tte)}$ and a finite part
$\gamma$ of $\phi(\ttp)$ such that $\alpha\subseteq\gamma$.
\item[ii.]
If and when i halts then output $\tty$.
\end{enumerate}
It is clear that $g$ is partial $A$-recursive and satisfies $(*)$.
\end{proof}

%
\subsection{Functionals over $\PR[\bbbx\to\bbby]$
            and computability}
\label{ss:computfunctionalsPR}

Using indexes, one can also consider computability for functionals
operating on the sole partial recursive or $A$-recursive functions.
\begin{definition}
Let $A\subseteq\bbbn$ and let
$(\varphi^{\bbbx\to\bbby,A}_\tte)_{\tte\in\words}$ denote
some acceptable enumeration of $\PR[A,\bbbx\to\bbby]$
(cf. Def.\ref{def:acceptable}).
\medskip\\
{\bf 1.}
A functional $\Phi:\PR[A,\bbbx\to\bbby]\times\bbbs\to\bbbt$
is an $A$-effective functional on partial $A$-recursive functions if
there exists some partial $A$-recursive function
$f:\words\to\words$ such that, for all $\tts\in\bbbs,\tte\in\words$,
$$\Phi(\varphi^{\bbbx\to\bbby,A}_\tte)=f(\tte)$$
We denote
$A\tiret\Eff^{\PR[A,\bbbx\to\bbby]\times\bbbs\to\bbbt}$
the family of such functionals.
\medskip\\
{\bf 2.}
We denote
$A\tiret\Eff^{\PR[A,\bbbx\to\bbby]\times\bbbs_1
\to \PR[A,\bbbs_2\to\bbbt]}$
the family of functionals obtained by curryfication of the above
class with $\bbbs=\bbbs_1\times\bbbs_2$.
\\
An easy application of the parameter property shows that these
functionals are exactly those for which there exists some partial
$A$-recursive function $g:\words\times\bbbs_1\to\words$ such that,
for all $\tts_1\in\bbbs_1,\tte\in\words$,
$$\Phi(\varphi^{\bbbx\to\bbby,A}_\tte,\tts_1)
=\varphi^{\bbbs_2\to\bbbt,A}_{g(\tte,\tts_1)}$$
\end{definition}
\begin{note}$\\ $
{\bf 1.}
Thanks to Rogers' theorem (cf. Thm.\ref{thm:rogers}), the above
definition does not depend on the chosen acceptable enumerations.
\medskip\\
{\bf 2.}
The above functions $f,g$ should have the following properties:
\begin{eqnarray*}
\varphi^{\bbbx\to\bbby,A}_\tte=\varphi^{\bbbx\to\bbby,A}_{\tte'}
&\Rightarrow&f(\tte,\tts)=f(\tte',\tts)
\\
\varphi^{\bbbx\to\bbby,A}_\tte=\varphi^{\bbbx\to\bbby,A}_{\tte'}
&\Rightarrow&
\varphi^{\bbbs_2\to\bbbt,A}_{g(\tte,\tts_1)}
=\varphi^{\bbbs_2\to\bbbt,A}_{g(\tte',\tts_1)}
\end{eqnarray*}
\end{note}
As shown by the following remarkable result, such functionals
essentially reduce to those of Def.\ref{def:Uspenskii}
(cf. Odifreddi's book \cite{odifreddi} p.206--208).
\begin{theorem}[Uspenskii \cite{uspenskii55}, Myhill \& Shepherdson
\cite{myhillshepherdson}]\label{thm:uspenskii}$\\ $
Let $A\subseteq\bbbn$.
The $A$-effective functionals
$\PR[A,\bbbx\to\bbby]\to \PR[A,\bbbs\to\bbbt]$
are exactly the restrictions to $\PR[A,\bbbx\to\bbby]$ of
$A$-effectively continuous functionals 
$(\bbbx\to\bbby)\to(\bbbs\to\bbbt)$.
\end{theorem}
%
%
\subsection{Effectivizations of Church representation of $\bbbn$}
\label{ss:effectiveChurch}
%
Observe the following trivial fact (which uses notations from
Def.\ref{def:Kleene},\ref{def:Uspenskii}).
\begin{proposition}\label{prop:chuchsystems}
Let $A\subseteq\bbbn$ and $\tau$ be any 2d order type.
\\
Functionals in $A\tiret\KleenePC^{\words\to\tau}$
(resp. $A\tiret\EffCont^{\words\to\tau}$)
are total maps $\words\to A\tiret\KleenePC^{\tau}$
(resp. $\words\to A\tiret\EffCont^{\tau}$).
\end{proposition}
\begin{theorem}\label{thm:chuchsystems}
Let $\tau$ be any 2d order type.
The systems
$$(A\tiret\KleenePC^\tau,A\tiret\KleenePC^{\words\to\tau})
\ \ \ ,
\ \ \ (A\tiret\EffCont^\tau,A\tiret\EffCont^{\words\to\tau})$$
are self-enumerated representation $A$-systems.
\end{theorem}
\begin{proof}
Points i-ii of Def.\ref{def:self} are trivial.
As for point iii, we use the classical enumeration theorem for
partial computable (resp. effectively continuous) functionals:
consider a function
$V\in A\tiret\KleenePC^{\words\to(\words\to\tau)}$ which
enumerates $A\tiret\KleenePC^{\words\to\tau}$ and set
$U(c(\tte,\ttp))=V(\tte)(\ttp)$.
Idem with $A\tiret\EffCont$.
\end{proof}
As an easy corollary of  Thms.\ref{thm:chuchsystems} and
\ref{thm:uspenskii}, we get the following result.

\begin{theorem}\label{thm:chuchsystems2}
Let $A\subseteq\bbbn$. Let
$A\tiret
\Eff^{\words\to (\PR[A,\bbbx\to\bbby]\times\bbbs\to\bbbt)}$
be obtained by curryfication from
$A\tiret \Eff^{(\PR[A,\bbbx\to\bbby]\times\bbbs\times\words)
\to\bbbt}$.
The systems
\begin{eqnarray*}
(A\tiret \Eff^{\PR[A,\bbbx\to\bbby]\times\bbbs\to\bbbt}&,&
A\tiret
\Eff^{\words\to (\PR[A,\bbbx\to\bbby]\times\bbbs\to\bbbt)})
\\
(A\tiret \Eff^{\PR[A,\bbbx\to\bbby]\to \PR[A,\bbbx\to\bbby]}&,&
A\tiret
\Eff^{\words\to (\PR[A,\bbbx\to\bbby]\to \PR[A,\bbbs\to\bbbt])})
\end{eqnarray*}
are self-enumerated representation $A$-systems.
\end{theorem}
\begin{definition}
[Effectivizations of Church representation of $\bbbn$]
\label{def:effChurch}
We effectivize the Church representation by replacing
$(X\to X)\to(X\to X)$ by one of the following
classes:
$$A\tiret\KleenePC^{(\bbbx\to\bbbx)\to(\bbbx\to\bbbx)}
\ ,\
A\tiret\EffCont^{(\bbbx\to\bbbx)\to(\bbbx\to\bbbx)}
\ ,\
A\tiret\Eff^{\PR[A,\bbbx\to\bbby]\to \PR[A,\bbbx\to\bbby]}$$
where $\bbbx$ is some basic set.
and $A\subseteq\bbbn$ is some oracle.
Using Def.\ref{def:church}, this leads to three self-enumerated
systems with domain $\bbbn$ :
\begin{eqnarray*}
{\cal F}_1&=&(\bbbn\ ,\
\Church^\bbbn_\bbbx\circ
A\tiret\KleenePC^{\words\to((\bbbx\to\bbbx)\to(\bbbx\to\bbbx))})
\\
{\cal F}_2&=&(\bbbn\ ,\
\Church^\bbbn_\bbbx\circ
A\tiret\EffCont^{\words\to((\bbbx\to\bbbx)\to(\bbbx\to\bbbx))})
\\
{\cal F}_3&=&(\bbbn\ ,\
\church^{\bbbn,A}_\bbbx\circ
A\tiret\Eff^{\words\to(\PR[A,\bbbx\to\bbby]
                   \to \PR[A,\bbbx\to\bbby])})
\end{eqnarray*}
\end{definition}
The following result greatly simplifies
the landscape.
\begin{theorem}\label{thm:allchurchequal}
The three systems ${\cal F}_1$, ${\cal F}_2$, ${\cal F}_3$ of
Def.\ref{def:effChurch} coincide.
\end{theorem}
Before proving the theorem (cf. the end of this subsection),
we state some convenient tools in the next three propositions,
the first of which will also be used in \S\ref{ss:syntaxChurch}.
\begin{proposition}\label{p:tool1}
Suppose $R\subset Fin(\bbbx\to\bbbx)\times\bbbx\times\bbbx$ is
functional (cf. Def.\ref{def:Uspenskii}).
The following conditions are equivalent
\begin{enumerate}
\item[i.]
$\Phi_R=It^{(n)}_\bbbx$
\item[ii.]
$\Phi_R\segment Fin(\bbbx\to\bbbx)
=It^{(n)}_\bbbx\segment Fin(\bbbx\to\bbbx)$
\item[iii.]
$\forall\alpha\in Fin(\bbbx\to\bbbx)\ \forall\ttx\
(\alpha^{(n)}(\ttx)\mbox{ is defined }\Rightarrow$

\hfill{$(\alpha\segment\{\alpha^{(i)}(\ttx):0\leq i<n\},
\ttx,\alpha^{(n)}(\ttx))\in R)$}
\\
and\\
$\forall\alpha\in Fin(\bbbx\to\bbbx)\ \forall\ttx\ \forall\tty$

\hfill{$((\alpha,\ttx,\tty)\in R\ \Rightarrow\
(\alpha^{(n)}(\ttx)\mbox{ is defined }\wedge\
\tty=\alpha^{(n)}(\ttx)))$}
\end{enumerate}
\end{proposition}
\begin{proof}
$iii\Rightarrow i$ and $i\Rightarrow ii$ are trivial.
\\
$ii\Rightarrow iii.$
Assume $ii$.
Suppose $(\alpha,\ttx,\tty)\in R$ then $\Phi_R(\alpha)(\ttx)=\tty$.
Since $\alpha\in Fin(\bbbx\to\bbbx)$, $ii$ insures that
$\alpha^{(n)}(\ttx)$ is defined and $\alpha^{(n)}(\ttx)=\tty$.
This proves the second part of $iii$.
\\
Suppose $\alpha^{(n)}(\ttx)$ is defined and let
$\alpha^{(n)}(\ttx)=\tty$. Then
\begin{eqnarray*}
\Phi_R(\alpha\segment\{\alpha^{(i)}(\ttx):0\leq i<n\})(\ttx)
&=&
It^{(n)}_\bbbx(\alpha\segment\{\alpha^{(i)}(\ttx):0\leq i<n\})(\ttx)
\\&=&It^{(n)}_\bbbx(\alpha)(\ttx)
\\&=&\tty
\end{eqnarray*}
So that there exists a restriction $\beta$ of
$\alpha\segment\{\alpha^{(i)}(\ttx):0\leq i<n\}$ such that
$(\beta,\ttx,\tty)\in R$. Thus, $\Phi_R(\beta)(\ttx)=\tty$.
Applying $ii$, this yields that $\beta^{(n)}(\ttx)$ is defined and
$\beta^{(n)}(\ttx)=\tty$.
Since $\beta$ is a restriction of
$\alpha\segment\{\alpha^{(i)}(\ttx):0\leq i<n\}$, this insures that
$\beta=\alpha\segment\{\alpha^{(i)}(\ttx):0\leq i<n\}$.
This proves the first part of $iii$.
\end{proof}
\begin{proposition}\label{p:tool2}
Let $n\in\bbbn$.
If $\Phi_R(f)$ is a restriction of $f^{(n)}$ for every
$f:\bbbx\to\bbbx$ then either $\Phi_R=It^{(n)}_\bbbx$ or $\Phi_R$
is not an iterator.
\end{proposition}
\begin{proof}
We reduce to the case $\bbbx=\bbbn$. Let $Succ:\bbbn\to\bbbn$ be
the successor function. Since $\Phi_R(Succ)$ is a restriction of
$Succ^{(n)}$, either $\Phi_R(Succ)(0)$ is undefined or
$\Phi_R(Succ)(0)=n$.
In both cases it is different from $Succ^{(p)}(0)$ for any
$p\neq n$. Which proves that $\Phi_R\neq It^{(p)}_\bbbn$ for
every $p\neq n$. Hence the proposition.
\end{proof}
\begin{proposition}\label{p:tool3}$\\ $
{\bf 1.}
Let $(W_\tte)_{\tte\in\words}$
be an acceptable enumeration of r.e. subsets of
$Fin(\bbbx\to\bbbx)\times\bbbx\times\bbbx$.
There exists a total recursive function $\xi:\words\to\words$
such that, for all $\tte$,
\begin{enumerate}
\item[a.]
$W_{\xi(\tte)}\subseteq W_\tte$ and $W_{\xi(\tte)}$ is functional
(cf. Def.\ref{def:Uspenskii}, point 1),
\item[b.]
$W_{\xi(\tte)}=W_\tte$ whenever $W_\tte$ is functional.
\end{enumerate}
{\bf 2.}
There exists a partial recursive function $\lambda:\words\to\bbbn$
such that
if $R_\tte$ is functional and $\Phi_{R_\tte}$ is an iterator
then $\lambda(\tte)$ is defined and
$\Phi_{R_\tte}=It^{(\lambda(\tte))}_\bbbx$.
(However, $\lambda(\tte)$ may be defined even if $R_\tte$ is not
functional or $\Phi_{R_\tte}$ is not an iterator).
\medskip\\
{\bf 3.}
There exists a total recursive function $\theta:\words\to\words$
such that, for all $\tte\in\words$,
\begin{enumerate}
\item[a.]
if $\Phi_{R_\tte}$ is an iterator then the $(\bbbx\to\bbbx)$-oracle
Turing machine ${\cal M}_{\theta(\tte)}$ with code $\theta(\tte)$
(cf. Def.\ref{def:Kleene}) computes the functional $\Phi_{R_\tte}$,
\item[b.]
if $\Phi_{R_\tte}$ is not an iterator then neither is the functional
computed by the $(\bbbx\to\bbbx)$-oracle Turing machine
${\cal M}_{\theta(\tte)}$ with code $\theta(\tte)$.
\end{enumerate}
In other words,
$Church(\Phi_{R_\tte})=Church(\Phi_{{\cal M}_{\theta(\tte)}})$
\medskip\\
{\bf 4.}
The above points relativize to any oracle $A\subseteq\bbbn$.
\end{proposition}
\begin{proof}
1. This is the classical fact underlying the enumeration theorem
for effectively continuous functionals.
To get $W_{\xi(\tte)}$, enumerate $W_\tte$ and retain a triple
if and only if, together with the already retained ones, it does not
contradict functionality
(cf. Odifreddi's book \cite{odifreddi} p.197).
\medskip\\
2. We reduce to the case $\bbbx=\bbbn$.
Let $\alpha_n:\bbbn\to\bbbn$ be such that
$$domain(\alpha_n)=\{0,...,n\}\ \ ,\ \ \alpha_n(i)=i+1\mbox{ for }
i=0,...,n$$
Suppose $R$ is functional and $\Phi_R=It^{(n)}_\bbbn$.
Prop.\ref{p:tool1} insures $(\alpha_n,0,n)\in R$.
\\
Also, for $m\neq n$, since $\alpha_m$ and $\alpha_n$ are
compatible and $R$ is functional, $R$ cannot contain
$(\alpha_m,0,m)$.
Thus, if $\Phi_R=It^{(n)}_\bbbn$ then $n$ is the unique integer
such that $R$ contains $(\alpha_n,0,n)$.
\medskip\\
This leads to the following definition of the wanted partial
recursive function $\lambda:\words\to\bbbn$ :
\medskip\\\indent- enumerate $R\tte$,
\\\indent- if and when some triple $(\alpha_n,0,n)$ appears,
halt and output $\lambda(\tte)=n$.
\medskip\\
3. Given a code $\tte$ of a functional relation $R_\tte$, we let
$\theta$ be the total recursive function which gives a code for the
oracle Turing machine ${\cal M}$ which acts as follows:
\begin{enumerate}
\item[i.]
First, it computes $\lambda(\tte)$.
\item[ii.]
If $\lambda(\tte)$ is defined then, on input $\ttx$ and oracle $f$,
${\cal M}$ tries to compute $It^{(\lambda(\tte))}_\bbbx(f)(\ttx)$
in the obvious way: ask the oracle the values of $f^{(i)}(\ttx)$ for
$i\leq\lambda(\tte)$.
\item[iii.]
Finally, in case i and ii halt, ${\cal M}$ enumerates $R\tte$ and
halts and accepts (with the output computed at phase ii)
if and only if
$(f\segment\{f^{(i)}(\ttx):i\leq\lambda(\tte)\},
\ttx,f^{(\lambda(\tte))}(\ttx))$ appears in $R_\tte$.
I.e. if and only if $f^{(\lambda(\tte))}(\ttx)=\Phi_R(f)(\ttx)$
\end{enumerate}
Clearly, the functional $\Phi_{\cal M}$ computed by ${\cal M}$
is such that $\Phi_{\cal M}(f)$ is equal to or is a restriction of
$It^{(\lambda(\tte))}_\bbbx(f)$.
\\
If $\Phi_{R_\tte}$ is an iterator then point 2 insures that
$\Phi_{R_\tte}=It^{(\lambda(\tte))}_\bbbx$ and Prop.\ref{p:tool1}
insures that phase iii is no problem, so that ${\cal M}$ computes
exactly $\Phi_{R_\tte}$.
\\
Suppose $\Phi_{R_\tte}$ is not an iterator.
\\
If $\lambda(\tte)$ is undefined then ${\cal M}$ computes the constant
functional with value the nowhere defined function. Thus, ${\cal M}$
does not compute an iterator.
\\
If $\lambda(\tte)$ is defined then, on input $\ttx$, ${\cal M}$ computes
$f^{(\lambda(\tte))}(\ttx)$ and halt and accepts if and only
$f^{(\lambda(\tte))}(\ttx)=\Phi_R(f)(\ttx)$.
Since $\Phi_R$ is not an iterator, there exists $f$ and $\ttx$
such that $f^{(\lambda(\tte))}(\ttx)$ is defined and
$\Phi_R(f)(\ttx)\neq f^{(\lambda(\tte))}(\ttx)$.
Hence $\Phi_{\cal M}(f)$ is a strict restriction of
$It^{(\lambda(\tte))}_\bbbx(f)$, so that
$\Phi_{\cal M}\neq It^{(\lambda(\tte))}_\bbbx$.
Finally, Prop.\ref{p:tool2} insures that $\Phi_{R_\tte}$ cannot be
an iterator.
\end{proof}
%
%
\noindent{\bf\em Proof of Theorem \ref{thm:allchurchequal}.}\\
1. Since $Fin(\bbbx\to\bbbx)\subset \PR[A,\bbbx\to\bbby]$,
condition $ii$ of Prop.\ref{p:tool1} and Thm.\ref{thm:uspenskii}
prove equality ${\cal F}_2={\cal F}_3$.
\medskip\\
2. Inclusion ${\cal F}_1\subseteq{\cal F}_2$ is a corollary of
Thm.\ref{thm:KleeneUspenskii}, point 1.
Let's prove the converse inclusion.
Suppose $\Phi:(\words\times(\bbbx\to\bbbx))\to(\bbbx\to\bbbx)$
is effectively continuous and let
$R\subseteq \words\times Fin(\bbbx\to\bbbx)\times\bbbx\times\bbbx$
be a functional r.e. set such that $\Phi=\Phi_R$.
Using the parameter property, let $h:\words\to\words$ be a total
recursive function such that $h(\tte)$ is an r.e. code for
$R^{(\tte)}=\{(\alpha,\ttx,\tty):(\tte,\alpha,\ttx,\tty)\in R\}$.
Prop.\ref{p:tool3}, point 3, gives a total recursive
$\theta:\words\to\words$ such that
$Church(\Phi_{R^{(\tte)}})=Church(\Phi_{{\cal M}_{\theta(\tte)}})$.
Thus, $\tte\mapsto Church(\Phi_{R^{(\tte)}})$ is partial computable
with a $(\bbbx\to\bbbx)$-oracle Turing machine having inputs in
$\words\times\bbbx$.
\hfill{$\Box$}
%
\subsection{Some examples of effectively continuous functionals}
\label{ss:examples}
%
For future use in sections
\S\ref{ss:syntaxChurch}-\ref{ss:characterizeChurch}, let's get the
following examples of effectively continuous functionals.
\begin{proposition}\label{p:examples}
If $\varphi:\words\to\bbbn$ is partial $A$-recursive and
$S\subseteq\words$ is $\Pi^{0,A}_2$ then there exists an
$A$-effectively continuous functional
$$\Phi:\words\to(\bbbx\to\bbbx)^{\bbbx\to\bbbx}$$
such that, for all $\ttp$,
\medskip\\
$\begin{array}{crcl}
(*)\hspace{1cm}&\ttp\in S\cap domain(\varphi)&\Rightarrow&
\Phi(\ttp)=It^{(\varphi(\ttp))}_\bbbx
\\
(**)\hspace{1cm}&\ttp\notin S\cap domain(\varphi)&\Rightarrow&
\Phi(\ttp)\mbox{ is not an iterator}
\end{array}$
\end{proposition}
\begin{proof}
We consider the sole case $A=\emptyset$, relativization being
straightforward.\\
Let $S=\{\tte:\forall u\ \exists v\ (\tte,u,v)\in\sigma\}$ where
$\sigma$ is a recursive subset of $\words\times\bbbn\times\bbbn$.
We construct a total recursive function $g:\words\to\words$ such
that, for all $\ttp$, $W_{g(\ttp)}$ is functional and
\begin{eqnarray*}
\ttp\in S\cap domain(\varphi)&\Rightarrow&
\Phi_{W_{g(\ttp)}}=It^{(\varphi(\ttp))}_\bbbx
\\
\ttp\notin S\cap domain(\varphi)&\Rightarrow&
\Phi_{W_{g(\ttp)}}\mbox{ is not an iterator}
\end{eqnarray*}
Let
\begin{eqnarray*}
{\cal S}_n&=&\{(\alpha,\ttx,\tty):\alpha\in Fin(\bbbx\to\bbbx)
\ \wedge\ \alpha^{(n)}(\ttx)\mbox{ is defined}
\ \wedge\ \tty=\alpha^{(n)}
\\
&&\hspace{5cm}\wedge\ domain(\alpha)=\{\alpha^{(i)}:i\leq n\}\}
\end{eqnarray*}

Let $\gamma:\bbbn^2\to\bigcup_{n\in\bbbn}{\cal S}_n$ be a
total recursive function such that, for all $n$,
$u\mapsto\gamma(n,u)$ is a bijection $\bbbn\to{\cal S}_n$.
Set
$$\rho_\tte=\{\gamma(\varphi(\tte),u):
\varphi(\tte)\mbox{ is defined}\ \wedge\ 
                      \exists v\ (\tte,u,v)\in\sigma\}$$
Clearly, $\rho_\tte$ is functional. Also, the construction of the
$\rho_\tte$'s is effective and the parametrization property yields
a total recursive function $g:\words\to\words$
such that $\rho_\tte=W_{g(\tte)}$.
\\
If $\varphi(\tte)$ is not defined then $\rho_\tte=\emptyset$ so that
$\Phi_{\rho_\tte}$ is the constant functional which maps any function
to the nowhere defined function. In particular, $\Phi_{\rho_\tte}$
is not an iterator.
\\
Suppose $\varphi(\tte)$ is defined.
Condition $iii$ of Prop.\ref{p:tool1} and the definition of
$\rho_\tte$ show that
\begin{eqnarray*}
\Phi_{\rho_\tte}\mbox{ is an iterator}
&\Leftrightarrow& \Phi_{\rho_\tte}=It^{(\varphi(n))}_\bbbx
\\
&\Leftrightarrow& \rho_\tte\supseteq
range(u\mapsto\gamma(\varphi(n),u))
\\
&\Leftrightarrow& \forall u\ \exists v\ (\tte,u,v)\in\sigma
\\
&\Leftrightarrow& \tte\in S
\end{eqnarray*}
Since $\rho_\tte=W_{g(\tte)}$, the functional
$\Phi:\tte\mapsto\Phi_{\rho_\tte}$ is  effectively continuous.
Clearly, it satisfies $(*)$ and $(**)$.
\end{proof}
%
\subsection{Syntactical complexity of Church representation}
\label{ss:syntaxChurch}
%
\begin{proposition}[Syntactical complexity]
\label{p:complexEffChurch}
The family
$$\{domain(\varphi):\varphi\in \Church^\bbbn_X\circ
A\tiret\EffCont^{\words\to((\bbbx\to\bbbx)\to(\bbbx\to\bbbx))}\}$$
is exactly the family of $\Pi^{0,A}_2$ subsets of $\words$.
\medskip\\
Thus, any universal function for
$\Church^\bbbn_X\circ
A\tiret\EffCont^{\words\to((\bbbx\to\bbbx)\to(\bbbx\to\bbbx))}$
has $\Pi^{0,A}_2$-complete domain.
\end{proposition}
\begin{proof}
To simplify notations, we only consider the case $A=\emptyset$.
Relativization being straightforward.
\\
1. Prop.\ref{p:examples} insures that every $\Pi^0_2$ set is
the domain of $Church^\bbbn_\bbbx\circ\Phi$ for some effectively
continuous functional $\Phi$.
\medskip\\
2. Conversely, we prove that every function in
$\Church^\bbbn_X\circ
\EffCont^{\words\to((\bbbx\to\bbbx)\to(\bbbx\to\bbbx))}$
has $\Pi^0_2$ domain.
\\
Suppose $\Phi:(\words\times(\bbbx\to\bbbx))\to(\bbbx\to\bbbx)$
is effectively continuous and let
$R\subseteq \words\times Fin(\bbbx\to\bbbx)\times\bbbx\times\bbbx$
be a functional r.e. set such that $\Phi=\Phi_R$.
For $\tte\in\words$, let
$R^\tte=\{(\alpha,\ttx,\tty):(\tte,\alpha,\ttx,\tty)\in R\}$.
Then
$$domain(\Church^\bbbn_X\circ\Phi)=\{\tte:\Phi_{R^\tte}
\mbox{ is an iterator}\}$$
Now, an r.e. code for the functional relation $R^\tte$ is
given by a total recursive function $h:\words\to\words$.
Applying Prop.\ref{p:tool2}, point 2, the partial recursive
function $\lambda\circ h$ is such that if $\Phi_{R^\tte}$
is an iterator then $\Phi_{R^\tte}=It^{(\lambda(h(\tte)))}_\bbbx$.
\\
Thus, $\Phi_{R^\tte}$ is an iterator if and only if
\begin{enumerate}
\item[a.]
$\lambda(h(\tte))$ is convergent,
\item[b.]
condition $iii$ of Prop.\ref{p:tool1} with $n=\lambda(h(\tte))$ holds.
\end{enumerate}
Condition a is $\Sigma^0_1$ and condition b is $\Pi^0_2$.
Thus, $domain(\Church^\bbbn_X\circ\Phi)$ is $\Pi^0_2$.
\end{proof}
%
%
\subsection{Characterization of the $\Church$ representation system}
\label{ss:characterizeChurch}
%
\begin{theorem}\label{thm:Church}
Let's denote $\PR[A,\words\to\bbbn]\segment\Pi^{0,A}_2$ the
family of restrictions to $\Pi^{0,A}_2$ subsets of
partial $A$-recursive functions $\words\to\bbbn$.
\\ Let $\bbbx$ be some basic set and $A\subseteq\bbbn$ be some
oracle.
\medskip\\
{\bf 1.}\ \ \
$\Church\circ
    A\tiret\EffCont^{\words\to((\bbbx\to\bbbx)\to(\bbbx\to\bbbx))}
=\PR[A,\words\to\bbbn]\segment\Pi^{0,A}_2$
\medskip\\
{\bf 2.}\ \ \
$K^\bbbn_{\Church\circ 
A\tiret\EffCont^{\words\to((\bbbx\to\bbbx)\to(\bbbx\to\bbbx))}}
\eqct K^A$
\medskip\\
We shall simply write $K^{\bbbn,A}_{\Church}$, or
$K^{\bbbn}_{\Church}$ when $A=\emptyset$.
\end{theorem}
\begin{proof}
1A. First, we prove that, for any $A$-effectively continuous
functional $\Phi:\words\to(\bbbx\to\bbbx)^{\bbbx\to\bbbx}$,
the function $\Church\circ\Phi:\words\to\bbbn$ has a partial
$A$-recursive extension.
We reduce to the case $\bbbx=\bbbn$.
\\
Let $Succ:\bbbn\to\bbbn$ be the successor function.
Observe that, for all $n\in\bbbn$,
$$(It^{(n)}_\bbbn(Succ))(0)=n$$
Thus, if $Church(\Phi(\tte)$ is defined then
$Church(\Phi(\tte))=(\Phi(\tte)(Succ))(0)$.
Applying Prop.\ref{p:apply}, we see that
$\tte\mapsto(\Phi(\tte)(Succ))(0)$ is a partial $A$-recursive
extension of $\Church\circ\Phi:\words\to\bbbn$.
\medskip\\
1B. Prop.\ref{p:complexEffChurch} insures that
$\Church\circ\Phi:\words\to\bbbn$ has $\Pi^{0,A}_2$ domain.
Together with point 1A, this insures that
$\Church\circ\Phi:\words\to\bbbn$ is the restriction of a partial
$A$-recursive function to a $\Pi^{0,A}_2$ set.
This proves the inclusion
$$\Church\circ
    A\tiret\EffCont^{\words\to((\bbbx\to\bbbx)\to(\bbbx\to\bbbx))}
\subseteq \PR[A,\words\to\bbbn]\segment\Pi^{0,A}_2$$
1C. The converse inclusion is Prop.\ref{p:examples}.
\medskip\\
2. Inclusion
$\PR[A,\words\to\bbbn]\subseteq\Church\circ
A\tiret\EffCont^{\words\to((\bbbx\to\bbbx)\to(\bbbx\to\bbbx))}$
yields the inequality 
$K^\bbbn_{\Church\circ 
A\tiret\EffCont^{\words\to((\bbbx\to\bbbx)\to(\bbbx\to\bbbx))}}
\leqct K^A$.
\medskip\\
Consider a function $\phi\in \Church\circ
A\tiret\EffCont^{\words\to((\bbbx\to\bbbx)\to(\bbbx\to\bbbx))}$.
Let $\widehat{\phi}$ be a partial $A$-recursive extension of $\phi$.
Then $K_\phi\geq K_{\widehat{\phi}}$.
This proves inequality
$K^\bbbn_{\Church\circ 
A\tiret\EffCont^{\words\to((\bbbx\to\bbbx)\to(\bbbx\to\bbbx))}}
\geqct K^A$.
\end{proof}
%
%
\subsection{Characterization of the $\Delta\Church$ self-enumerated
systems}  \label{ss:DeltaKChurch}

\begin{theorem}\label{thm:DeltaChurch}
Let $\bbbx$ be some basic set and $A\subseteq\bbbn$ be some oracle.
\medskip\\
{\bf 1.}\ \ \
$\Delta(\Church\circ
    A\tiret\EffCont^{\words\to((\bbbx\to\bbbx)\to(\bbbx\to\bbbx))})
=\PR[A,\words\to\bbbz]\segment\Pi^{0,A}_2$
\medskip\\
{\bf 2.}\ \ \
$K^\bbbz_{\Delta(\Church\circ
    A\tiret\EffCont^{\words\to((\bbbx\to\bbbx)\to(\bbbx\to\bbbx))})}
\eqct K^A_\bbbz$
\medskip\\
We shall simply write $K^{\bbbz,A}_{\Delta\Church}$, or
$K^{\bbbz}_{\Delta\Church}$ when $A=\emptyset$.
\end{theorem}
\begin{proof}
1. Observe that $\Delta(\PR[A,\words\to\bbbn]\segment\Pi^{0,A}_2)
=\PR[A,\words\to\bbbz]\segment\Pi^{0,A}_2$ and apply
Thm.\ref{thm:Church}.
\medskip\\
2. Argue as in point 2 of the proof of Thm.\ref{thm:Church}.
\end{proof}
%
%
%
\subsection{Functional representations of $\bbbz$}
\label{ss:churchZ}
%
Specific to Church representation, there is another approach
for an extension to $\bbbz$ : {\em positive and negative iterations}
of injective functions over some infinite set $X$. Formally,
I.e., letting $X\stackrel{1-1}{\to}X$ denote the family of injective
functions, consider the $\bbbz$-iterator functional
$$It^\bbbz_X:\bbbz
\to(X\stackrel{1-1}{\to}X)^{X\stackrel{1-1}{\to}X}$$
such that, for $n\in\bbbn$, $It^\bbbz_X(n)(f)=f^{(n)}$
and $It^\bbbz_X(-n)(f)=It^\bbbz_X(n)(f^{-1})$.
\\
Effectivization can be done as in \S\ref{ss:effectiveChurch}.
Thm.\ref{thm:allchurchequal}, Prop.\ref{p:complexEffChurch} and
Thm.\ref{thm:Church} go through the $\bbbz$ context.
%

%
\section{Conclusion}\label{s:conclusion}
%
%
We have characterized Kolmogorov complexities associated to some
set theoretical representations of $\bbbn$ in terms of the
Kolmogorov complexities associated to oracular and/or infinite
computations (Thm.\ref{thm:A}).
As a corollary, we got a hierarchy result (Thm.\ref{thm:B}).
\medskip\\
These results can be improved in two directions.
\\
First, one can consider higher order (higher than type 2)
effectivizations of set theoretical representations of $\bbbn$.
This is the contents of a forthcoming continuation of this paper.
\\
Second, using the results of our paper \cite{ferbusgrigoOrder},
the hierarchy result Thm.\ref{thm:B} can be improved with finer
orderings than $\infct$.
These orderings $\lless[{\cal C},{\cal D}]{\cal F}$ are
such that $f\lless[{\cal C},{\cal D}]{\cal F}g$ if and only if
\begin{enumerate}
\item
$f\leqct g$
\item
For every infinite set $X\in{\cal C}$ and every total monotone
increasing function $\phi\in{\cal F}$ there exists an infinite set
$Y\in{\cal D}$ such that
\\\centerline{$Y\subseteq\{z\in X:f(z)<\phi(g(x))\}$}
\item
The above property is effective: relative to standard enumerations
of ${\cal C},{\cal D},{\cal F}$, a code for $Y$ can be recursively
computed from codes for $X$ and $\phi$.
\end{enumerate}
Thm.\ref{thm:B} can be restated in the following improved form.
\begin{theorem}\label{thm:Abis}
Denote $\minpr[]$ (resp. $\minprjump[A]$) the family of functions
$\bbbn\to\bbbn$ which are infima of partial recursive
(resp.partial $A$-recursive) sequences of functions $\bbbn\to\bbbn$
(cf. Rk.\ref{rk:minPR}).
Then
\medskip\\
$\log \ggreater[\Sigma^0_1,\Sigma^0_1]{\PR[]}
\begin{array}{c}
    K_{\Church}^\bbbn\\
    \eqct\\
    K_{\Church}^\bbbz\segment\bbbn\\
    \eqct\\
    K_{\Delta \Church}^\bbbz\segment\bbbn
\end{array}
\ggreater[\Sigma^0_1\cup\Pi^0_1,\Delta^0_2]{\minpr[]}
K_{\card}^\bbbn
\ggreater[\Sigma^0_2,\Sigma^0_2]{\PR[\emptyset']}
K_{\Delta \card}^\bbbz\segment\bbbn$

\hfill{$\ggreater[\Sigma^0_2\cup\Pi^0_2,\Delta^0_3]
{\minprjump[\emptyset']}
K_{\indexx}^\bbbn
\ggreater[\Sigma^0_3,\Sigma^0_3]{\PR[\emptyset'']}
K_{\Delta \indexx}^\bbbz\segment\bbbn$}
\end{theorem}

%

\end{document}